\documentclass[a4paper,reqno,11pt]{amsart}
\usepackage[T1]{fontenc}
\usepackage[utf8]{inputenc}
\usepackage[english]{babel}

\usepackage{graphicx}
\usepackage{amsmath}
\usepackage{amsthm}
\usepackage{amssymb}
\usepackage{mathtools}
\usepackage{amsmath}
\usepackage{amsfonts}
\usepackage{mathrsfs}
\usepackage{esint}
\usepackage{yfonts}
\usepackage{quoting}
\usepackage{graphics}
\usepackage{booktabs}
\usepackage{bm}
\usepackage{bbm}
\usepackage{lmodern}
\usepackage{stmaryrd}
\usepackage{caption}
\usepackage{comment}
\usepackage{enumitem}			
\usepackage{tikz}
\usepackage{bookmark}
\usepackage{pgfplots}
\usepgfplotslibrary{fillbetween}
\usepackage{orcidlink}
\usepackage[mathscr]{euscript}

\usepackage{hyperref}

\usepackage{cleveref}

\hypersetup{
	colorlinks=true,
	linkcolor=blue,
	citecolor=blue,
	urlcolor=black,
	linktoc=all
}

\usepackage[margin=1.1in]{geometry}

\quotingsetup{font=small}

\theoremstyle{plain}
\newtheorem{proposition}{Proposition}[section]

\theoremstyle{plain}
\newtheorem{theorem}{Theorem}[section]
\numberwithin{equation}{section}	

\theoremstyle{plain}
\newtheorem{lemma}{Lemma}
\numberwithin{lemma}{section}

\theoremstyle{plain}

\theoremstyle{definition}
\newtheorem{remark}{Remark}[section]

\newtheorem{example}{Example}

\DeclareMathOperator{\sign}{sign}
\DeclareMathOperator{\inj}{inj}

\DeclareMathOperator{\divg}{div}
\DeclareMathOperator{\ric}{Ric}

\newcommand{\p}{\partial}
\newcommand{\pphi}{\varphi}
\newcommand{\eps}{\varepsilon}
\newcommand{\R}{\mathbb{R}}
\newcommand{\Sp}{\mathbb{S}^{n-1}}
\newcommand{\Om}{\Omega}
\newcommand{\al}{\alpha}
\newcommand{\bb}{\beta}
\newcommand{\ot}{\otimes}

\newcommand{\lap}[1]{\Delta #1}
\newcommand{\lapw}[1]{\Delta_g^f #1}

\def\XXint#1#2#3{{\setbox0=\hbox{$#1{#2#3}{\int}$ }
		\vcenter{\hbox{$#2#3$ }}\kern-.6\wd0}}

\usepackage{enumitem}

\begin{document}
	\title[Serrin's Problem on weighted Riemannian Manifolds]{Symmetry and rigidity results for Serrin's overdetermined type problems in weighted Riemannian manifolds}
	
	\author[Laura Accornero]{Laura Accornero} 
\address[]{Laura Accornero. Dipartimento di Matematica ‘Federigo Enriques’, Università degli Studi di Milano, Via Cesare Saldini 50, 20133, Milan, Italy}
\email{laura.accornero@unimi.it}

\author[Giulio Ciraolo]{Giulio Ciraolo \orcidlink{0000-0002-9308-0147}}
\address[]{Giulio Ciraolo. Dipartimento di Matematica ‘Federigo Enriques’, Università degli Studi di Milano, Via Cesare Saldini 50, 20133, Milan, Italy}
\email{giulio.ciraolo@unimi.it}

\date{\today}
\dedicatory{}

\begin{abstract}
We study Serrin's overdetermined boundary value problems in bounded domains on weighted Riemannian manifolds. When the closure of the domain is compact, we establish a rigidity result that characterizes both the solution and the geometry of the ambient manifold. We further address the case of domains with non-compact closure for manifolds conformally equivalent to the Euclidean space, possibly degenerating or becoming singular at a point, where both the weight and the conformal factor are radial functions. 
\end{abstract}

\maketitle

\maketitle
\tableofcontents

\section{Introduction}
The classical overdetermined boundary value problem introduced by Serrin in his seminal paper \cite{Serrin} concerns solutions $u$ of the torsion problem
\begin{equation}\label{serrin1}
	\begin{cases}
		\Delta u = -1 & \text{in } \Omega, \\
		u = 0 & \text{on } \partial \Omega,
	\end{cases}
\end{equation}
supplemented with the additional condition 
\begin{equation}\label{serrin2}
	\frac{\partial u}{\partial \nu} = -c \quad \mbox{on} \ \partial \Omega \, .
\end{equation} 
Here, $\Omega \subset \R^n$, with $n \geq 2$, is a bounded domain, $\nu$ denotes the outward the normal derivative on $\partial \Omega$, and $c$ is a positive constant, and hence it is given by 
\begin{equation}\label{c_def}
	c= \frac{|\Omega|}{|\partial \Omega|} \,,
\end{equation}
as it follows by a straightforward application of the divergence theorem. 

Serrin's celebrated theorem states that the existence of such a solution forces $\Omega$ to be a ball and $u$ to be radially symmetric. In particular, one has that $\Omega = B_R(x_0)$ for some $R>0$ and $x_0 \in \R^n$ and
\begin{equation} \label{u_torsion}
	u(x)=\frac{R^2-|x-x_0|^2}{2n} \,.
\end{equation}

The main tool of Serrin's proof is a technique known as \textit{moving planes}, inspired by the proof of Alexandrov's Soap Bubble Theorem (see \cite{Alexandrov}). This approach turns out to be very powerful, as it can be used to generalize Serrin' symmetry result to more general semilinear and quasilinear equations (see \cite{Serrin}).  

Right after Serrin's paper, Weinberger in \cite{Weinberger} provided a very short proof of the same result for the torsion problem \eqref{serrin1}, using an integral identity (namely, the Pohozaev identity) and the maximum principle applied to an auxiliary function, later called $P$-function. Over the years, further extensions and new techniques have appeared in the literature, including variational, geometric, and fully nonlinear generalizations (see, e.g., \cite{BianchiniCiraolo2018,BrandoliniNitschSalaniTrombetti,ChoulliHenrot,CiraoloRoncoroni2018,FarinaKawohl,FarinaValdinoci,FigalliZhang,FragalaGazzolaKawohl,FragalaVelichkov,GarofaloLewis,RWH}). Among these, we mention \cite{BianchiniCiraolo2018} and \cite{BrandoliniNitschSalaniTrombetti} where Serrin's overdetermined problem is proved by using integral identities.

The approaches based on integral identities, as well as Weinberger's approach with the $P$-function, may give information not only on the classification of the solution but also on the ambient manifold, once the problem is settled in a Riemannian manifold.
Indeed, these methods rely on the introduction of the $P$-function or a suitable vector field which combines the gradient and the values of $u$. Then, an application of Pohozaev's identity gives that the Hessian $\nabla^2 u$ is a multiple of the identity matrix at every point and hence the solution $u$ must be radial. This directly leads to the conclusion that $\Omega$ must be a ball.

The occurrence of functions whose Hessian is proportional to the identity is a classical argument appearing also in geometric analysis, particularly in rigidity phenomena. For instance, it is well-known from Tashiro \cite{Tashiro} that if a complete Riemannian manifold admits a nontrivial function whose Hessian is a constant multiple of the metric, then the manifold must be isometric to the Euclidean space. Thus, the structural conclusion arising in the overdetermined problem mirrors well-known geometric rigidity results.

It is therefore natural to view Serrin's problem not only as a classification result for solutions of an elliptic equation, but also a way to obtain rigidity theorems in Riemannian geometry. 

A first step in this direction has been done by Roncoroni \cite{Roncoroni} and Farina \& Roncoroni \cite{FarinaRoncoroni} who studied Serrin's overdetermined problem
\begin{equation}\label{serrinknu}
	\begin{cases}
		\Delta{u}+nku=-1 & \mbox{in} \ \Omega\\
		u=0 & \mbox{on}  \ \partial \Omega \\
		\frac{\partial u}{\partial \nu} = -c & \mbox{on} \ \partial \Omega \, ,
	\end{cases}
\end{equation} 
on warped product manifolds with Ricci curvature bounded from below. More recently, Freitas, Roncoroni \& Santos \cite{FreitasRoncoroniSantos} have studied problem \eqref{serrinknu} on manifolds endowed with a closed conformal vector field. Finally, Andrade, Freitas \& Mar\'{i}n \cite{AndradeFreitasMarin} considered the case of some specific class of manifolds endowed with a conformal vector field. 

In this paper, we aim to study Serrin's overdetermined problem in weighted manifolds, where the geometry is modified by the presence of a smooth density and the Laplacian and Ricci tensor are replaced by the weighted Laplacian and the Bakry--\'Emery Ricci tensor, respectively. In this setting, both the analytic behavior of solutions and the geometric interpretation of rigidity become subtler. Few partial results have been obtained (see for instance \cite{AraujoFreitasSantos} and \cite{RWH}) and a complete analogue of Serrin's theorem -- along with its geometric implications -- remains open. 
\subsection{The weighted setting and two motivating examples}
We recall that a weighted Riemannian manifold is defined as a triple $(M,g,d\mu_g)$, where $(M,g)$ is a Riemannian manifold with Riemannian measure $dV_g$, $e^{-f}:M\rightarrow \R$ is a smooth function, which denotes the weight, and $d\mu_g=e^{-f}dV_g$ is the weighted measure. The weighted Laplacian operator is defined by
\begin{equation}\label{operator}
	\lapw u \coloneq e^f \divg_g(e^{-f} \nabla u) = \lap_g u - g(\nabla f,\nabla u) \, ,
\end{equation}
where $$\lap_g u =\divg_g \nabla u$$ is the Laplacian of $u$.\\
Moreover, the $m$-dimensional Bakry-\'{E}mery-Ricci tensor is given by
\begin{equation}\label{Ricmf}
	\ric_f^m\coloneq \ric+\nabla^2f-\frac {df\otimes df}{m-n} \, ,
\end{equation}
where $m \in \R\setminus\{0\}$ and $\ric$ is the Ricci tensor of $(M,g)$. Finally, the $\infty$-dimensional is defined by
\begin{equation}\label{Ricinf}
	\ric_{f}^{\infty}\coloneq \ric+\nabla^2f \, .
\end{equation}
The weighted volume of a domain $\Om\subset M$ is defined as
$$|\Om|_f \coloneq \int_\Om d\mu_g=\int_\Om e^{-f} \, dV_g \, .$$
Throughout, we denote by $\overline \Om^g$ the closure of $\Om$ with respect to the metric $g$.
\medskip

We start by discussing two examples, which exhibit several intriguing feature that make them particularly interesting and are a motivation for this paper.

\begin{example} \label{ex_power}
	We consider the triple $(\R^n,\delta, |x|^\alpha dx)$, where $\delta$ is the standard Euclidean metric and $\alpha >-n$. Thus, in this case, $e^{-f}=|x|^\alpha$ and
	\begin{equation*}
		f(x) = -\alpha \log |x| \,.
	\end{equation*}
	
	It is straightforward to see that a solution to 
	\begin{equation} \label{torsion_ex}
		\begin{cases} 
			\lapw u = -1 & \textmd{ in } B_R(0) \\
			u=0  & \textmd{ on } \partial B_R(0)
		\end{cases}
	\end{equation}
	is given by 
	\begin{equation} \label{u_torsion_ex}
		u^*(x) = \frac{R^2 - |x|^2}{2(n+\alpha)} \,,
	\end{equation}
	and clearly \eqref{u_torsion_ex} fulfills 
	\begin{equation*}
		u^*_\nu = -\frac{R}{n+\alpha} \ \textmd{ on } \partial B_R(0) \quad \textmd{ and } \quad \nabla^2 u^* = - \frac{1}{n+\alpha} I  \ \textmd{ in }  B_R(0) \,.
	\end{equation*}   
	Hence, it is natural to consider a Serrin's type symmetry result for $(\R^n,\delta, |x|^\alpha dx)$, and whether this overdetermined problem leads also to a rigidity theorem for some related weighted manifolds. 
	
	Since rigidity results are often obtained under a nonnegativity assumption on a suitable Ricci tensor, it may be useful to understand for which manifolds it is reasonable to expect a symmetry result, it is important to understand for which $m$ we have that $\ric_f^m \geq 0$. Straightforward calculations give
	\begin{equation*}
		f_i = - \alpha \frac{x_i}{|x|^2} \quad \textmd{ and } f_{ij} =  - \frac{\alpha}{|x|^2}\left(\delta_{ij} - 2 \frac{x_i x_j}{|x|^2} \right)
	\end{equation*}
	and 
	\begin{equation*}
		\ric_f^m =  - \frac{\alpha}{|x|^2}\left(\delta_{ij} + \left(\frac{\alpha}{m-n} - 2\right) \frac{x_i x_j}{|x|^2} \right)  
	\end{equation*}
	where we set $\dfrac{\alpha}{m-n} = 0 $ if $m=\infty$. Thus, if $\alpha \neq 0$ then
	\begin{equation}\label{Ricex1}
		\ric_f^m \geq 0 \quad \iff \quad  \alpha < 0 \ \textmd{ and } \ n + \alpha \leq m \leq n \,.
	\end{equation}
	We will return to this condition later in the introduction.
\end{example}

\begin{example}
	Serrin's solution to \eqref{serrin1} -- \eqref{serrin2} is closely related to an isoperimetric
	type inequality for the so-called torsional rigidity. For a bounded connected domain $\Om \subset \R^n$ with
	smooth boundary $\p\Om$, the torsional rigidity is defined by
	\begin{equation*}
			\tau(\Om)\coloneq \sup_{v \in W_{0}^{1,2}(\Om)\setminus\{0\}}\frac{\displaystyle\left(\int_\Om v \, dx\right)^2}{\displaystyle\int_\Om |D v|^2 \, dx} \, .
	\end{equation*}
	By the direct method of Calculus of Variations, the supremum of $\tau(\Om)$ is achieved by a multiple
	of the function $u$ satisfying \eqref{serrin1}. The isoperimetric problem for $\tau(\Om)$ is answered by Saint Venant’s principle, which states that $\tau(\Om) \leq \tau(B_r)$, where $B_r$ is a ball such
	that $|\Om| = |B_r|$, with equality holding if and only if $\Om= B_r$.  The standard proof of
	this result uses rearrangement techniques and isoperimetric inequality (see \cite[Section 1.12]{PoylaSzego}). \\
	On the other hand, from the first variational formula or Hadamard’s formula for $\tau(\Om)$ it follows that that a stationary domain for $\tau(\Om)$ among domains with fixed volume must satisfy $\p_\nu u = -c$. Therefore, Serrin’s symmetry result implies that the only stationary domains for $\tau(\Omega)$ under fixed volume are balls.\\
	\indent Now, consider the torsion problem for the weighted operator
	\begin{equation*} 
		\begin{cases} 
			\lapw u = -1 & \textmd{ in } \Om \\
			u=0  & \textmd{ on } \partial \Om \, ,
		\end{cases}
	\end{equation*}
	where $\Om$ is a bounded domain in $\R^n$ and $\delta$ is the Euclidean metric. The weighted torsional rigidity of $\Om$ is defined as 
	$$\tau^f(\Om)\coloneq \sup_{v \in W_{0}^{1,2}(\Om)\setminus\{0\}}\frac{\displaystyle\left(\int_\Om v \, e^{-f} \, dx\right)^2}{\displaystyle\int_\Om |D v|^2 \, e^{-f} \, dx} \, .$$
Exploiting the results on the log-convex conjecture in \cite{Chambers}, in \cite{BrandoliniChiacchio} the authors consider smooth, concave, radial functions $f$ (equivalently, log-convex radial weights $e^{-f}$) and prove via weighted rearrangement techniques that, among domains with fixed weighted volume, the centered ball maximizes the weighted torsional rigidity.\\
	As in the classical setting, Hadamard’s formula for $\tau^f(\Omega)$ shows that a stationary domain under a weighted volume constraint must satisfy the overdetermined boundary condition $\p_\nu u = -c$. This naturally leads to the question of whether a Serrin-type symmetry result holds in the weighted Euclidean space $(\mathbb{R}^n, \delta, e^{-f}dx)$ for this class of radial weights.
\end{example}
\subsection{Main results}
In this paper, we prove two main results.
\par In Theorem \ref{main}, we establish a rigidity result for both the solution and the underlying manifold associated with \eqref{ourprob} in weighted Riemannian manifolds when the closure of the domain $\overline\Om^g$ is compact.
In \cite{AndradeFreitasMarin,FarinaRoncoroni,FreitasRoncoroniSantos,Roncoroni}, Serrin-type results for \eqref{serrinknu} are obtained via the $P$-function method. In contrast, we adopt a different approach based on integral identities, namely a Pohozaev-type identity adapted to our setting, which may be of independent interest. For completeness, we also provide an alternative proof based on the classical $P$-function method and compare the two approaches. In both cases, the classification of the ambient manifold follows from a local version of Tashiro’s theorem.

Note that the $P$-function method relies on the strong maximum principle and requires $\overline\Om^g$ to be compact; hence, it is not directly applicable when this condition fails. Our second result addresses this issue. In particular, in Theorem \ref{main_nc} we consider a class of weighted Riemannian manifolds for which $\overline\Om^g$ may not be compact, so that the $P$-function method is not applicable a priori, while our approach still applies. More precisely, we extend Serrin’s result to manifolds conformally equivalent to the Euclidean space, possibly degenerating or becoming singular at a point, where both the weight and the conformal factor are radial functions. However, due to the presence of a singularity, no rigidity result for the ambient manifold is obtained in this case.

In what follows, for simplicity, we will refer to the two settings as the compact case and the non-compact case, corresponding to whether the closure of the domain $\overline{\Omega}^g$ is compact or not.
\par Our first main result is the following:
\begin{theorem}\label{main}
	Let $(M,g,e^{-f}dV_g)$ be a weighted, Riemannian manifold of dimension $n$, with $e^{-f}\in C^2(M)$. Let $\Om$ be a bounded, open and connected set in $(M,g,e^{-f}dV_g)$ with boundary $\p\Om$ of class $C^{1}$. Assume that 
	\begin{itemize}
		\item [\textit(i)] there exist $\beta \in (0,1]$ and $V\in C^2(\Om)\cap C^1(\overline \Om^g)$ such that 
		\begin{equation} \label{V_hp}
			\nabla^2 V\leq \frac \beta n g \quad \text{and} \quad \lapw V=1 \quad \text{in } \Om\, ;
		\end{equation}
		\item [\textit(ii)] it holds 
		\begin{equation}\label{Ricf_hp}
			\ric_f^{N} \geq 0 \quad \text{in } \Om
		\end{equation}
		for some $N \leq N_\beta$, where 
		\begin{equation}
			N_\beta =
			\begin{cases}
				\frac{n}{\beta} & \textmd{ if } 0< \beta < 1 \\
				\infty& \textmd{ if } \beta = 1 \,;
			\end{cases}
		\end{equation}
		\item [\textit(iii)] if $\beta=1$, then
		\begin{equation}\label{iii}
			g(\nabla f,\nabla u)\leq 0 \quad \text{in} \ \Om \, .
		\end{equation}
	\end{itemize}
	If $u$ is a smooth solution to
	\begin{equation}\label{ourprob}
		\begin{cases}
			\lapw{u}=-1 & \mbox{in} \ \Omega\\
			u=0 & \mbox{on}  \ \partial \Om \\
			u_\nu = -c & \mbox{on} \ \partial \Om \, , 
		\end{cases}
	\end{equation}
	then $u$ is radial and $\Om$ is a metric ball $B_R^g(p)$ isometric to a Euclidean ball, $\beta=1$, $e^{-f}$ is constant and $u$ is given by
	\begin{equation}\label{uespl}
		u(x)=\frac{R^2-d_g(x,p)^2}{2n}\, .
	\end{equation}
\end{theorem}
\hfill\medskip\\
\noindent We now comment the assumptions of Theorem \ref{main}. 

The existence of a function $V$ satisfying \eqref{V_hp} is motivated as follows. Since we aim to prove a rigidity result based on a Serrin's overdetermined problem, it is reasonable to assume that at least the Hessian of the solution of the torsion problem in a ball is comparable to the metric. In particular, in the model examples of Example \ref{ex_power}, we can take $V=-u^*$ and hence \eqref{V_hp} is fulfilled by setting $\beta=\frac{n}{n+\alpha}. $

Condition \eqref{Ricf_hp} appears natural within the framework of weighted Riemannian manifolds. For instance, in \cite{AraujoFreitasSantos} the authors established a Serrin-type rigidity result for the overdetermined problem \eqref{serrinknu} in weighted generalized cones by imposing the curvature condition
$$\ric_{f}^{\alpha+n}\geq k(n+\alpha-1)g \, ,$$
where $e^{-f}$ is a homogeneous weight of degree $\alpha>0$. In our setting, this corresponds to the case $k=0$. Even in the non-weighted framework, several rigidity results have been obtained under a lower bound assumption on the Ricci curvature (see \cite{AndradeFreitasMarin}, \cite{FarinaRoncoroni}, \cite{FreitasRoncoroniSantos}).\\
\indent Finally, the compactness of $\overline\Om^g$ allow us to apply \Cref{farinaroncoroni} and derive a rigidity result for the underlying manifold. In \Cref{domaincpt} we show that this assumption is in fact necessary. 
\par \Cref{main} shows that, if $\overline\Om^g$ is compact, then there are no non-trivial weights satisfying conditions \eqref{V_hp} and \eqref{Ricf_hp} for which problem \eqref{ourprob} admits a solution. Unfortunately, if $M$ is $\R^n$ endowed with the Euclidean metric $\delta$, weights of the form 
\[
e^{-f} = |x|^\alpha
\] 
with $\alpha \neq 0$ do not fall into this class. Indeed, as noted in \Cref{ex_power}, in this case a necessary condition for \eqref{Ricf_hp} to hold is \(\alpha < 0\), while $\alpha\geq 0$ is necessary if we require the weight to be continuous. Consequently, rigidity for problem \eqref{ourprob} in the setting of \Cref{ex_power} remains a challenging open problem.
\par Let $O$ denote the origin of $\R^n$. In the second result of this paper, we consider the manifold 
\[
M = \mathbb{R}^n \setminus \{O\}
\] 
endowed with a Riemannian metric conformal to the Euclidean one, and we show that, in this setting, such weights can be recovered without forcing \(\alpha = 0\).
\\ Specifically, we consider the overdetermined problem \eqref{ourprob} in the weighted Riemannian manifold 
$$
(\R^n\setminus\{O\},g,e^{-f}dV_g)
$$
with the conformal metric
\begin{equation}\label{g}
	g_{ij}=e^{2\phi}\delta_{ij}\coloneqq|x|^{-2\gamma}\delta_{ij} \, ,
\end{equation}
where $\gamma<1$. As we will see in \Cref{Ofinitedist}, this condition ensures that the origin is not a point at infinity, that is, the geodesic distance from $x$ to $O$, defined by
\begin{equation*}
	d_g(x,O)\coloneq \inf\{L(\sigma) : \sigma:[0,1)\to M \ \text{admissible curve}, \sigma(0)=x, \lim_{t \to 1^-} \sigma(t) = O\} \, ,
\end{equation*}
where $L(\sigma)$ denotes the length of the curve $\sigma$ with respect to the metric $g$, is finite for every $x \in M$. In this case, we say that a function $u$ is radially symmetric with respect to the origin if it depends only on $d_g(x,O)$, and we also define the ball of radius $R$ centered at the origin with respect to $g$ as
\begin{equation}\label{defball}
	B_R^g \coloneqq \{ x \in M : d_g(x,O) < R \} \,.
\end{equation} 
We consider the case of homogeneous radial weights 
\begin{equation} \label{pesopower}
	e^{-f(x)}=d_g(x,O)^\alpha \, ,
\end{equation} 
where $\alpha \in \R$. In this setting, we will show that, for a suitable range of $\alpha$ and $\gamma$, conditions \eqref{V_hp} and \eqref{Ricf_hp} are satisfied for this class of weights, with $\beta<1$. Hence, in this case, we can obtain a Serrin-type symmetry result similarly to the compact case. Moreover, as noted in \Cref{gammamin1}, the assumption $\gamma<1$ is not restrictive: even for $\gamma>1$, we can obtain an analogous result for the overdetermined problem \eqref{ourprob}.\\
We consider problem \eqref{ourprob} in weak sense, i.e., $u$ is a weak solution if it satisfies
\begin{equation}\label{weak}
	\int_\Om g(\nabla u,\nabla \varphi)\,d\mu_g = \int_\Om \varphi \, d\mu_g \quad \forall \varphi \in C^\infty_{c} (\Om) \, ,
\end{equation}
where 
$$d\mu_g \coloneqq e^{-f}dV_g \, .$$ 
Thus, from \eqref{weak}, we observe that the natural space related to the problem \eqref{ourprob} is \\$W_0^{1,2}(\Om,d\mu_g)$, defined as the closure of $C^\infty_c(\Om)$ with respect to the norm given by 
\begin{equation}\label{weightnorm}
	\|u\|_{W^{1,2}_0(\Om,d\mu_g)}\coloneqq \left(\int_\Om u^2 \, d\mu_g  +\int_\Om |\nabla u|^2 \, d\mu_g \right)^{\frac 1 2} \, .
\end{equation}
As we will see in \Cref{preliminarynoncompl}, this choice is natural and justified by technical considerations related to the divergence theorem.
\begin{theorem}\label{main_nc}
	Let $n\geq 3$, $\Om$ be a bounded, open and connected set in $(\R^n\setminus\{O\},g)$ with compact boundary $\p\Om$ of class $C^{1,s}$, and let $u \in W^{1,2}_0(\Om,d\mu_g)\cap L^\infty(\Om)$ be a solution to \eqref{ourprob} with weight as in \eqref{pesopower} such that 
	\begin{equation}\label{regassu}
		|\nabla u|_g^2 \in L^\infty(\Om) \quad \text{and} \quad \nabla |\nabla u|_g^2 \in L^2(\Om,d\mu_g) \, .
	\end{equation}
	Define 
	\begin{equation*}
		\gamma_1 \coloneqq 1-\sqrt{\frac{n-2}{\alpha+n-2}} \, .
	\end{equation*}
	If \begin{equation}\label{parcond}
		\alpha>0 \quad \text{and} \quad \gamma_1 \leq \gamma <1 \, ,
	\end{equation} 
	then $u$ is radially symmetric, it is given by
	\begin{equation}\label{sol_nc}
		u(x)=\frac{R^2-d_g(x,O)^2}{2(n+\alpha)} \quad \text{in} \ \Om \, ,
	\end{equation}
	and $\Om$ is a ball centered at the origin of radius $R$ with respect to the metric $g$.
\end{theorem}
\begin{remark}
	Note that the range of admissible parameters in \Cref{main_nc} is always non-degenerate, since $0<\gamma_1<1$ for any $\alpha>0$. Moreover, the compactness of $\p\Om$ ensures that the singularity of the metric $g$ at the origin does not lie on the boundary of $\Om$. Therefore, since $\p\Om \in C^{1,s}$, classical regularity theory implies that $u \in C^\infty(\Om)\cap C^1(\overline \Om^g)$. 
\end{remark}
\begin{remark}\label{domaincpt}
	Note that, since the origin does not belong to the manifold, balls centered at the origin are not compact in $(\mathbb{R}^n \setminus \{O\}, g)$. From \Cref{main_nc}, we have $\Omega = B_R^g$ for some $R>0$. Hence, \Cref{farinaroncoroni} cannot be applied to classify the manifold $M$, since it leads to a contradiction. 
	Indeed, an application of \Cref{farinaroncoroni} implies that $\Omega$ is isometric to a Euclidean ball, which in turn leads to $\gamma = 0$. Since this value does not belong to the range of parameters specified in \eqref{parcond}, this results in a contradiction: no solution $u$ to \eqref{ourprob} can satisfy the assumptions of \Cref{main_nc}. However, \eqref{sol_nc} actually provides a solution to \eqref{ourprob} satisfying the assumptions of \Cref{main_nc}.\\
	This shows that \Cref{farinaroncoroni} cannot be applied in this case. In particular, since the solution $u$ given by \eqref{sol_nc} satisfies all the assumptions of \Cref{farinaroncoroni}, we conclude that the compactness of $\overline{\Omega}^g$ is not only sufficient, but also necessary to obtain a rigidity result for the manifold.
\end{remark}
\par\textbf{Organization of the paper.} The organization of the paper is as follows. In \Cref{compact_case}, we address the compact case. We begin by proving some preliminary results, including properties of solutions to \eqref{ourprob} and a new Pohozaev-type identity. We then prove \Cref{main}. For completeness, we also provide an alternative proof based on the $P$-function method and compare the two approaches.

In \Cref{nccompact_case}, we consider the non-compact case. We first adapt the preliminaries from the compact case to this setting, including a version of the divergence theorem suitable for bounded and non-compact domains. We then prove \Cref{main_nc} using an integral identity approach and discuss the non-applicability of the $P$-function method. Finally, we present a counterexample to \Cref{main_nc}, by providing a Riemannian metric of the form \eqref{g} for which the overdetermined problem \eqref{ourprob} does not have a metric ball as solution.

In \Cref{appendix}, we provide a proof of a local version of Tashiro's theorem, which is used to characterize both the solution and the ambient manifold in the compact case.
\subsection{Notations}\label{notations}	
Given a function $u :\R^n \rightarrow \R$, we denote by $Du=(\p_i u)_{i=1}^d$ and $D^2u=(\p_{ij} u)_{i,j=1}^d$ the Euclidean gradient and the Euclidean Hessian of $u$, respectively. The Euclidean Laplacian of $u$ is denoted by $\lap u$.\\ Let $|\cdot|$ be the Euclidean norm and $a \cdot b$ the scalar product between two points $a$ and $b$ of $\R^n$. Given a vector field $F:\R^n \rightarrow \R^n$, $\divg F=\p_i F^i$ is the divergence of $F$. From now on, we will adopt the Einstein summation convention over repeated indices.
\par Let $(M,g,e^{-f}dV_g)$ be a weighted Riemannian manifold of dimension $n$ and let $u$ be a smooth function on $M$. We denote by $\nabla u = (\nabla^i u)_{i=1}^n$ the Riemannian gradient of $u$. If $F$ is a smooth vector field on $M$, its divergence is defined as
$$\divg_g F = \frac{1}{\sqrt{|g|}}\p_i(\sqrt{|g|}F^i) \, .$$ The Laplacian of $u$ is given by
$$\lap_g u = \divg_g \nabla u $$
and the Hessian of $u$, when seen as a (0,2)-tensor, is defined as follows
$$(\nabla^2 u)_{ij} = \p_{ij}u-\Gamma_{ij}^k\p_k u \, ,$$
where $\Gamma_{ij}^k$ are the Christoffel symbols of the Levi-Civita connection associated with the Riemannian metric $g$. Given a (0,2)-tensor field $Y$ and a vector field $X$, we have that tr$_g(Y)=g^{ij}y_{ij}$ where $Y=(y_{ij})_{i,j=1,\dots,n}$ and $g^{-1}=(g^{ij})_{i,j=1,\dots,n}$, and we set
$$|Y|_g^2 = \mbox{tr}_g(y_{ij}^2)=g^{ij}y_{ij}^2$$
and
\begin{equation*}
	|X|_g^2=g(X, X) \, .
\end{equation*}
\par Let $k\in\mathbb{R}$. We denote by $S^n_k$ the $n$-dimensional space forms of constant sectional curvature $k$. Recall that
$$
S_k^n = \left([0,R_k)\times \Sp,\, dr^2 + \mathrm{sn}_k^2(r)\, g_{\Sp}\right),
$$
where $g_{\Sp}$ denotes the standard round metric on $\Sp$,
$$
\mathrm{sn}_k(r)\coloneq
\begin{cases}
	\dfrac{1}{\sqrt{k}}\sin(\sqrt{k}\,r) & \text{if } k>0,\\[6pt]
	r & \text{if } k=0,\\[6pt]
	\dfrac{1}{\sqrt{-k}}\sinh(\sqrt{-k}\,r) & \text{if } k<0 \, ,
\end{cases}
$$
and
\begin{equation}\label{Rk}
	R_k \coloneq
	\begin{cases}
		\dfrac{\pi}{\sqrt{k}} & \text{if } k>0,\\
		+\infty & \text{if } k\leq 0.
	\end{cases}
\end{equation}
Moreover, $S^n_k$ is the round sphere of radius $1/\sqrt{k}$ if $k>0$, the Euclidean space $\mathbb{R}^n$ if $k=0$, and the hyperbolic space of constant curvature $k$ if $k<0$.
\par Finally, we recall that, for every \( x \in \mathbb{R}^n \), there exists a natural identification of the tangent space \( T_x \mathbb{R}^n \) with \(\mathbb{R}^n\) via the map
\[
x \mapsto x^i \partial_i \,.
\]
Accordingly, the symbol \( x \) will sometimes denote a point in \(\mathbb{R}^n\) and sometimes the corresponding vector field; its meaning will always be clear from the context.
\section{The compact case}\label{compact_case}
\subsection{Preliminary results}
In this section, we establish several preliminary results that will be essential for the proof of \Cref{main}. 
We begin by proving two useful properties of solutions to the problem \eqref{ourprob}. 
\begin{lemma}\label{propsolutions}
	Let $\Om$ be a bounded domain of class $C^1$ and let $u$ be a smooth solution to \eqref{ourprob}.
	Then it holds
	\begin{equation}\label{cweight}
		c=\frac {|\Om|_f}{|\p \Om|_f} \, ,
	\end{equation}
	and
	\begin{equation}\label{intnablau}
		\int_\Om |\nabla u|_g^2 \,d\mu_g = \int_\Om u \,d\mu_g \, .
	\end{equation}
\end{lemma}
\begin{proof}
	Recalling that $u_\nu=-c$ on $\p\Om$, and applying the divergence theorem, we obtain
	$$|\Om|_f=\int_\Om e^{-f}\,dV_g=-\int_\Om \divg_g(e^{-f}\nabla u)\,dV_g=-\int_{\p\Om} e^{-f}g(\nabla u,\nu)\,d\tilde V_g = c\,|\partial\Om|_f \, ,$$
	from which we deduce \eqref{cweight}. Moreover, recalling that $\lapw u=-1$ in $\Om$, another application of the divergence theorem yields
	\begin{align*}
		\int_\Om u e^{-f}\,dV_g &=-\int_\Om u e^{-f} \lapw u \,dV_g\\
		&=-\int_\Om u\divg_g(e^{-f}\nabla u)\,dV_g\\
		&=\int_\Om g(\nabla u,\nabla u)e^{-f}\,dV_g-\int_{\p\Om}ue^{-f}g(\nabla u,\nu)\,d\tilde V_g \, ,
	\end{align*}
	from which we conclude \eqref{intnablau}, since $u=0$ on $\p\Om$. 
\end{proof}
We now establish a Pohozaev-type integral identity that will play a crucial role in the proof of Theorem~\ref{main}.
Recall the classical Pohozaev identity associated with Serrin’s overdetermined problem \eqref{serrin1} -- \eqref{serrin2}:
\begin{equation}\label{pohoclassica}
	\frac{n-2}{2} \int_\Omega |Du|^2 \, dx
	+ \frac{c^2}{2} \int_{\partial \Omega} x \cdot \nu \, d\sigma
	= n \int_\Omega u \, dx \, ,
\end{equation}
which follows by integrating over $\Om$ the differential identity
\begin{equation}\label{pohodiff}
	\divg\left(\frac{|Du|^2}{2}\, x - (x \cdot Du)\, Du\right)
	= \frac{n-2}{2}\, |Du|^2 - (x \cdot Du)\, \Delta u 
\end{equation}
and applying the divergence theorem. Hence, it is clear that the position vector $x$ plays a crucial role in the Euclidean setting. In the following we obtain a Pohozaev-type integral identity on a Riemannian manifold, where the role of the position vector $x$ is replaced by the gradient of $V \in C^2(\Om)\cap C^1(\overline \Om^g)$, with $V$ such that 
$$\lapw V=1 \quad \text{in } \Om \, ,$$
as we see below. We mention that this approach differs from \cite{FreitasRoncoroniSantos} and \cite{AraujoFreitasSantos}, where the authors addressed this issue by assuming the existence of a closed conformal vector field (see also \cite{AndradeFreitasMarin}, where the closedness assumption was removed).
\begin{proposition}\label{pohozaevprop}
	Let $\Om\subset M$ be a bounded domain of class $C^1$ and $u$ a smooth solution to \eqref{ourprob}. Let $V\in C^2(\Om)\cap C^1(\overline \Om^g)$ such that $\lapw V =1$ in $\Om$. Then
	\begin{equation}\label{pohozaev}
		c^2|\Om|_f =\int_\Om u \, d\mu_g +2\int_\Om \nabla^2V(\nabla u,\nabla u) \, d\mu_g \, .
	\end{equation}
\end{proposition}
\begin{proof}
	The proof follows as in the usual Pohozaev identity. We begin by noting that the divergence theorem yields
	\begin{equation}\label{poho2}
		\int_{\p\Om} g(\nabla V,\nu) \, d\tilde\mu_g =\int_\Om \divg_g (e^{-f}\nabla V) \, dV_g = \int_\Om \lapw V \, d\mu_g \, .
	\end{equation}
	Next, we compute
	\begin{equation}\label{poho3}
		\int_\Om g(\nabla V,\nabla u) \, d\mu_g =\int_\Om \divg_g(e^{-f}u\nabla V)\, dV_g-\int_\Om e^{-f}u \, \lapw V \, dV_g= -\int_\Om u \, \lapw V \, d\mu_g \, ,
	\end{equation}
	where we used the fact that $u=0$ on $\p\Om$. Now we consider the following differential identity
	\begin{equation}
		\begin{aligned}\label{diffpoho}
			\begin{split}
				&\divg_g\left(e^{-f}\frac{|\nabla u|_g^2}{2}\nabla V-e^{-f} g(\nabla V,\nabla u)\nabla u\right)\\
				&=e^{-f}\left\{\frac{|\nabla u|_g^2}{2}\lapw V+g(\nabla \left(\frac{|\nabla u|_g^2}{2}\right),\nabla V)-g(\nabla V,\nabla u)\lapw u-g(\nabla(g(\nabla V,\nabla u)),\nabla u)\right\} \quad \text{in } \Om\, .
			\end{split}
		\end{aligned}
	\end{equation}
	A straightforward computation shows that
	$$g(\nabla \left(\frac{|\nabla u|_g^2}{2}\right),\nabla V)-g(\nabla(g(\nabla V,\nabla u)),\nabla u) =-\nabla^2V(\nabla u,\nabla u)\, ,$$
	and thus, using $\lapw u=-1$ in $\Om$, \eqref{diffpoho} simplifies to
	\begin{equation}\label{divthmpoho1}
		\divg_g\left(e^{-f}\frac{|\nabla u|_g^2}{2}\nabla V-e^{-f} g(\nabla V,\nabla u)\nabla u\right)\\
		=e^{-f}\left\{\frac{|\nabla u|_g^2}{2}\lapw V+g(\nabla V,\nabla u)-\nabla^2 V(\nabla u,\nabla u)\right\} \, .
	\end{equation}
	By integrating over $\Om$, applying the divergence theorem and using the fact that $u=0$ on $\p\Om$, we obtain
	$$\int_{\p\Om} g(\frac{|\nabla u|_g^2}{2}\nabla V- g(\nabla V,\nabla u)\nabla u,\nu) \, d\tilde\mu_g =\int_\Om \left\{\frac{|\nabla u|_g^2}{2}\lapw V+g(\nabla V,\nabla u)-\nabla^2 V(\nabla u,\nabla u)\right\} \, d\mu_g \, .$$
	Recalling that $\nabla u=-c\nu$ on $\p\Om$, we get
	\begin{equation}\label{poho1}
		-\frac{c^2}{2}\int_{\p\Om} g(\nabla V,\nu) \, d\tilde\mu_g =\int_\Om \frac{|\nabla u|_g^2}{2}\lapw V \, d\mu_g +\int_\Om g(\nabla V,\nabla u) \, d\mu_g -\int_\Om \nabla^2V(\nabla u,\nabla u) \, d\mu_g \, .
	\end{equation}
	Hence, substituting \eqref{poho2} and \eqref{poho3} into \eqref{poho1}, we deduce that 
	$$-\frac{c^2}{2}\int_\Om \lapw V \, d\mu_g =\int_\Om \frac{|\nabla u|_g^2}{2}\lapw V \, d\mu_g -\int_\Om u \, \lapw V \, d\mu_g -\int_\Om \nabla^2 V(\nabla u,\nabla u) \, d\mu_g \, .$$
	Since $\lapw V=1$, we have that
	$$-\frac{c^2}{2}|\Om|_f =\frac 1 2 \int_\Om |\nabla u|_g^2 \, d\mu_g -\int_\Om u \, d\mu_g -\int_\Om \nabla^2 V(\nabla u,\nabla u) \, d\mu_g \, .$$
	Recalling from \Cref{propsolutions} that 
	$$\int_\Om u\, d\mu_g =\int_\Om |\nabla u|_g^2 \, d\mu_g \, ,$$
	we get
	$$-\frac{c^2}{2}|\Om|_f =-\frac 1 2 \int_\Om u \, d\mu_g -\int_\Om \nabla^2 V(\nabla u,\nabla u) \, d\mu_g \, ,$$
	which concludes the proof.
\end{proof}
The proof of \Cref{main} is based on the Pohozaev identity \eqref{pohozaev} and on the following differential identity, which was proved in \cite[Proposition 3.2]{CiraoloCorso} 
\begin{lemma}
	Let $u$ be a $C^3$ function. It holds
	\begin{equation}
		\begin{split}
			&e^f\divg_g \left[e^{-f}\left(u\nabla \left(\frac{|\nabla u|_g^2}{2}\right)-u\,\lapw u \nabla u-\frac 1 2 |\nabla u|_g^2\nabla u\right)\right]\\
			&=u\left[|\nabla ^2u|_g^2-(\lapw u)^2\right]-\frac 3 2 |\nabla u|^2\lapw u+u\,(\ric_g +\nabla ^2f)[\nabla u,\nabla u].
		\end{split}\label{idciraolo}
	\end{equation}
\end{lemma}
\noindent Moreover, we establish an Obata-type result, which we will use directly to deduce the rigidity of both the solutions and the domain of \eqref{ourprob}. We defer its proof to \Cref{appendix}.
	\begin{lemma}\label{farinaroncoroni}
	Let $(M,g)$ be an $n$-dimensional Riemannian manifold and let $\Om$ be a domain in $M$ such that $\overline\Om^g$ is compact. Assume that there exists $u \in C^0(\overline\Om^g)\cap C^2(\Om)$ solution to 
	\begin{equation}\label{eqhessu}
		\begin{cases}
			\nabla^2 u = -\left(\frac 1 n +ku\right)g & \text{in} \ \Om \, ,\\
			u>0 & \text{in} \ \Om \, ,\\
			u=0 & \text{on} \ \p\Om \, ,
		\end{cases}
	\end{equation}
	where $k \in \R$. Then $u$ is radial and $\Om$ is a metric ball $B_R^g(p)$. Moreover, $\Om$ is isometric to a metric ball in the space form $S_k^n$, and $u$ is given by
	\begin{equation}\label{ufin}
		u(r)=\begin{cases}
			\frac{\cos(\sqrt{k}r)}{kn \cos(\sqrt{k}R)}-\frac{1}{nk} & \text{if } k>0\, ,\\
			\frac{R^2-r^2}{2n} & \text{if } k=0 \, ,\\
			\frac{\cosh(\sqrt{-k}r)}{kn\cosh(\sqrt{-k}R)}-\frac{1}{nk} & \text{if } k<0 \, .
		\end{cases} 
	\end{equation}
\end{lemma}
A combination of the generalized Bochner identity, the Cauchy-Schwarz inequality and the elementary inequality 
\begin{equation}\label{cs1}
	(s+t)^2\geq \frac{s^2}{1+\alpha}-\frac{t^2}{\alpha} \quad \alpha>0 \,, \quad s,t \in \R \, ,
\end{equation}
yields the following Bochner-type inequality, as stated in \cite{Ledoux}. It  will play a crucial role in the proof of \Cref{main} via the $P$-function method and, for the reader’s convenience, we report its proof.
\begin{lemma}\label{weighted_Bochner}
	Let $(M,g,e^{-f}dV_g)$ be a weighted Riemannian manifold, $u \in C^3(M)$ and $m>n=dim(M)$. Then
	\begin{equation}\label{weighted_Bochnineq}
		\frac 1 2 \lapw |\nabla u|_g^2 \geq \frac{(\lapw u)^2}{m}+g(\nabla u,\nabla \lapw u)+\ric_f^m(\nabla u,\nabla u) \, ,
	\end{equation}
	where equality holds if and only if
	\begin{equation}\label{bochneq}
		\nabla^2 u=\frac{\lap_g u}{n}g
	\end{equation}
	and
	\begin{equation}\label{bochneq1}
		\lapw u=-\frac{m}{m-n} g(\nabla f,\nabla u) \, .
	\end{equation}
\end{lemma}
\begin{proof}
	Let 
	$$\Gamma_2(u,u)\coloneq \frac 1 2 \lapw |\nabla u|_g^2 -g(\nabla \lapw u,\nabla u) \, .$$
	From Bakry \cite{Bakry}, the generalized Bochner-Weitzenb\"{o}ck formula holds
	\begin{equation}\label{bochnweitz}
		\Gamma_2(u,u)=|\nabla^2 u|_g+\ric_f^\infty(\nabla u,\nabla u) \, ,
	\end{equation}
	where $\ric_{f}^\infty$ is given by \eqref{Ricinf}.
	Since 
	\begin{equation}\label{cs}
		|\nabla^2 u|_g\geq \frac 1 n (\lap_g u)^2 \, ,
	\end{equation}
	and using the elementary inequality \eqref{cs1}
	with $s=\lapw u$ and $t=g(\nabla f,\nabla u)$,
	from \eqref{bochnweitz} we obtain
	\begin{align*}
		\Gamma_2(u,u)&\geq \frac 1 n (\lap_g u)^2+\ric_f^\infty(\nabla u,\nabla u)\\
		&=\frac 1 n (\lapw u+g(\nabla f,\nabla u))^2+\ric_f^\infty(\nabla u,\nabla u)\\
		&\geq \frac{1}{n(1+\alpha)}(\lapw u)^2 - \frac{(df \otimes df)(\nabla u,\nabla u)}{n\alpha} +\ric_f^\infty(\nabla u,\nabla u) \, ,
	\end{align*}
	where we have identified $df\otimes df(\nabla u,\nabla u)$ with $g(\nabla f,\nabla u)^2$.
	Setting $m\coloneq n(1+\alpha)$, we conclude that 
	$$\Gamma_2(u,u)\geq \frac 1 m (\lapw u)^2 +\ric_f^m(\nabla u,\nabla u) \,,$$
	where $\ric_f^m$ is given by \eqref{Ricmf}.
	Since $\alpha>0$ is arbitrary, by the definition of $\Gamma_2(u,u)$ follows \eqref{weighted_Bochnineq}. Moreover, equality holds if and only if equality is attained in both \eqref{cs} and \eqref{cs1}, yielding \eqref{bochneq} and \eqref{bochneq1}, respectively.
\end{proof}
\subsection{Proof of \Cref{main}}
\hfill\\We are now ready to prove \Cref{main}.
\begin{proof}[Proof of \Cref{main}]
	Note that, since $\lapw u=-1$ in $\Om$ and $u=0$ on $\p\Om$, the strong maximum principle implies that $u$ is positive in $\Om$. Moreover, we observe that in $\Om$ it holds
	\begin{align}
		& |\nabla ^2u|_g^2-(\lapw u)^2 \nonumber\\
		&= |\nabla ^2 u|_g^2-(\lap_g u-g(\nabla f,\nabla u))^2\nonumber\\
		&=\left|\nabla ^2 u-\frac{\lap_g u}{n}g\right|_g^2+2\,g(\nabla f,\nabla u)\lap_g u-g(\nabla f,\nabla u)^2-\frac{n-1}{n}(\lap_g u)^2\, . \nonumber 
	\end{align}
	By using the identity
	$$\lap_g u= \lapw u+g(\nabla f,\nabla u) \, ,$$
	we deduce that
	\begin{equation}\label{id}
		\begin{aligned}
			& |\nabla ^2u|_g^2-(\lapw u)^2\\
			&=\left|\nabla ^2 u-\frac{\lap_g u}{n}g\right|_g^2-\frac{n-1}{n}(\lapw u)^2 +\frac 2 n \,g(\nabla f,\nabla u)\lapw u+\frac 1 n \, g(\nabla f,\nabla u)^2 \, .
		\end{aligned}
	\end{equation}
	Thus, integrating identity \eqref{idciraolo} and using \eqref{id}, we obtain
	\begin{equation*}
		\begin{split}
			&\int_\Om \divg_g \Big\{e^{-f}\left(u\nabla \left(\frac{|\nabla u|_g^2}{2}\right)-u \, \lapw u \nabla u-\frac{|\nabla u|_g^2}{2}\nabla u\right)\Big\}\,dV_g\\
			&=\int_\Om e^{-f}u\,\Big\{|\nabla ^2 u-\frac{\lap_g u}{n}g|_g^2-\frac{n-1}{n}(\lapw u)^2+\frac 2 n \, g(\nabla f,\nabla u)\lapw u +\frac 1 n \, g(\nabla f,\nabla u)^2 \Big\}\, dV_g\\
			&+\int_\Om e^{-f}\Big\{-\frac 3 2 |\nabla u|_g^2 \, \lapw u + u \, (\ric_g+\nabla ^2 f)[\nabla u,\nabla u]\Big\}\, dV_g \, .
		\end{split}
	\end{equation*}
	Since $u$ is such that $\lapw u=-1$ in $\Om$ and $u=0$ on $\p\Om$, from the divergence theorem we find
	\begin{align}
		-\int_{\p\Om} e^{-f} \frac{|\nabla u|_g^2}{2}\, g(\nabla u,\nu)\,d\tilde V_g &= \int_\Om e^{-f} \Big( -\frac{n-1}{n}\,u+\frac 3 2 |\nabla u|_g^2\Big)\, dV_g \nonumber\\
		&\begin{multlined}[t]+\int_\Om e^{-f} u \, \Big\{ \Big|\nabla ^2 u-\frac{\lap_g u}{n}g\Big|_g^2 +\frac 2 n \,g(\nabla f,\nabla u)\lapw u \\+\frac 1 n \,g(\nabla f,\nabla u)^2 + (\ric_g+\nabla ^2 f)[\nabla u,\nabla u]\Big\}\, dV_g \, . \label{prima}
		\end{multlined} 
	\end{align}
	Using \eqref{cweight} and the fact that $u_\nu =-c$ on $\p\Om$, the left-hand side of \eqref{prima} simplifies to $\frac {c^2}{2}|\Om|_f$. Moreover, from \eqref{intnablau}, we have that
	\begin{align*}
		\int_\Om \Big( -\frac{n-1}{n}u+\frac 3 2 |\nabla u|_g^2 \Big) \, d\mu_g = \frac{n+2}{2n}\int_\Om u\,d\mu_g \, .
	\end{align*}
	Substituting into \eqref{prima}, we obtain
	\begin{equation}\label{seconda}
		\begin{multlined}
			\frac{c^2}{2}|\Omega|_f = \frac{n+2}{2n} \int_\Omega u \, d\mu_g \\
			+ \int_\Omega u \, \Big\{ 
			\left|\nabla^2 u - \frac{\Delta_g u}{n} g \right|_g^2 
			+ \frac{2}{n} g(\nabla f, \nabla u)\lapw u
			+ \frac{1}{n} g(\nabla f, \nabla u)^2 
			+ (\mathrm{Ric}_g + \nabla^2 f)[\nabla u, \nabla u] 
			\Big\} \, d\mu_g \, . 
		\end{multlined}
	\end{equation}
	Next, applying the Pohozaev identity \eqref{pohozaev}, we get 
	\begin{equation}\label{postpoho}
		\begin{multlined}
			0=\frac 1 n \int_\Om u \, d\mu_g - \int_\Om \nabla^2 V(\nabla u,\nabla u) \, d\mu_g \\+ \int_\Omega u \, \Big\{ 
			\left|\nabla^2 u - \frac{\Delta_g u}{n} g \right|_g^2 
			+ \frac{2}{n} g(\nabla f, \nabla u)\lapw u
			+ \frac{1}{n} g(\nabla f, \nabla u)^2 
			+ (\mathrm{Ric}_g + \nabla^2 f)[\nabla u, \nabla u] 
			\Big\} \, d\mu_g \, . 
		\end{multlined}
	\end{equation}
	Using the assumption
	\begin{equation}\label{hessV}
		\nabla^2V\leq \frac \beta n g \, ,
	\end{equation}
	we obtain
	$$\int_\Om \nabla^2V(\nabla u,\nabla u)\, d\mu_g \leq \frac \beta n\int_\Om |\nabla u|_g^2 \, d\mu_g =\frac \beta n \int_\Om u \, d\mu_g \, ,$$
	where the last equality follows from \eqref{intnablau}. Hence, from \eqref{postpoho} we get 
	\begin{equation}\label{princ}
		\begin{multlined}
			0\geq \left(\frac{1-\beta}{n}\right)\int_\Om u \, d\mu_g  \\+ \int_\Omega u \, \Big\{ 
			\left|\nabla^2 u - \frac{\Delta_g u}{n} g \right|_g^2 
			+ \frac{2}{n} g(\nabla f, \nabla u)\lapw u
			+ \frac{1}{n} g(\nabla f, \nabla u)^2 
			+ (\mathrm{Ric}_g + \nabla^2 f)[\nabla u, \nabla u] 
			\Big\} \, d\mu_g \, . 
		\end{multlined}
	\end{equation}
	Now, we want to bound the right-hand side of \eqref{princ} from below by zero. For this, we distinguish between the cases $0<\beta < 1$ and $\beta = 1$.\smallskip\\
	\indent \textbf{Case} $\bm{0<\beta<1:}$ 
	By applying the weighted Young inequality $$2|ab|\leq k a^2+\frac 1 k b^2 \quad \text{with} \quad a=\frac{1}{\sqrt n}\lapw u \ \text{and} \ b=\frac{1}{\sqrt n} g(\nabla f,\nabla u) \quad \text{in} \ \Om \, ,$$ we find
	\begin{equation}\label{youngw}
		\frac 2 n \, g(\nabla f,\nabla u)\lapw u \geq -\frac k n (\lapw u)^2-\frac{1}{kn}g(\nabla f,\nabla u)^2 \quad \text{in} \ \Om\, ,
	\end{equation}
	where $k>0$ will be chosen appropriately. Note that the equality is attained if and only if $k \cdot \lapw u = -g(\nabla f,\nabla u)$, i.e., $k=g(\nabla f,\nabla u$).\\
	Plugging \eqref{youngw} into  \eqref{princ}, we obtain
	\begin{equation}\label{postyoung}
		\begin{split}
			0 &\geq \left(\frac{1-k-\beta}{n}\right)\int_\Om u\, d\mu_g+\int_\Om u\, \Big|\nabla ^2 u-\frac{\lap_g u}{n}g\Big|_g^2 \, d\mu_g \\
			&+\int_\Om u\, \Big\{\frac 1 n\Big(1-\frac{1}{k}\Big)\,g(\nabla f,\nabla u)^2+(\ric_g +\nabla^2f)[\nabla u,\nabla u]\Big\}\,d\mu_g \, . 
		\end{split}
	\end{equation}
	Now, we choose $k > 0$ such that 
	$$\frac{1-k-\beta}{n}=0 \, ,$$
	that is
	\begin{equation}\label{k}
		k=1-\beta \, .
	\end{equation}
	Recall that $k$ has to be positive, as it represents the coefficient used in the Young inequality \eqref{youngw}. In the case under consideration, we have $\beta<1$, so this condition is indeed satisfied. From \eqref{postyoung}, and observing that $$\frac 1 n \left(1-\frac 1 k\right)=-\frac{\beta}{n(1-\beta)} \quad \text{and} \quad g(\nabla f,\nabla u)^2=(df \otimes df) [\nabla u,\nabla u] \, ,$$
	we get
	\begin{equation}\label{ultima}
		\begin{aligned}
			0 \geq  \int_\Om u \,\Big|\nabla ^2 u-\frac{\lap_g u}{n}g\Big|_g^2 \, d\mu_g  
			+\int_\Om u \, (\ric_g +\nabla^2f-\frac{\beta}{n(1-\beta)} \, df \otimes df)[\nabla u,\nabla u]\,d\mu_g \, . 
		\end{aligned}
	\end{equation}
	Note that, since $u$ is positive in $\Om$ and 	
	\begin{equation}\label{inhessu}
		\Big |\nabla^2 u -\frac{\lap_g u}{n}g\Big |_g \geq 0\,,
	\end{equation}
	the first integral on the right hand side in \eqref{ultima} is nonnegative. Recall that we have assumed  
	\begin{equation*}
		\ric_g +\nabla ^2f-\frac{\beta}{n(1-\beta)} \, df \otimes df 
		\geq 0 \, ,
	\end{equation*} 
	which implies that
		\begin{equation}\label{ricwgeq0}
		(\ric_g +\nabla ^2f-\frac{\beta}{n(1-\beta)} \, df \otimes df)[\nabla u,\nabla u]\geq 0 \, .
	\end{equation} 
	This enables us to estimate both the integral terms on the right hand side of  \eqref{ultima} from below by zero. Recalling that $u>0$ in $\Om$, this implies that equality is attained in all the inequalities \eqref{hessV}, \eqref{youngw}, \eqref{inhessu} and \eqref{ricwgeq0}. From \eqref{youngw}, \eqref{inhessu} and \eqref{ricwgeq0}, we find respectively
	\begin{equation}\label{scalarprod}
		g(\nabla f,\nabla u)=1-\beta\,,
	\end{equation}
	\begin{equation}\label{hessu}
		\nabla ^2 u=\frac{\lap_g u}{n}g \, ,
	\end{equation}
	and
	\begin{equation*}
		(\ric_g +\nabla^2f-\frac{\beta}{n(1-\beta)} \, df \otimes df)[\nabla u,\nabla u]= 0\, .
	\end{equation*}
	Recalling that $\lapw u=-1$ and 
	$$\lap_g u=\lapw u+g(\nabla f,\nabla u) \, ,$$
	from \eqref{scalarprod} and \eqref{hessu} we obtain
	\begin{equation}\label{hessuespl}
		\nabla^2 u =- \frac \beta n g \, .
	\end{equation}
	Thus, by applying \Cref{farinaroncoroni} to $u/\beta$ with $k=0$, we conclude that $\Om$ is a metric ball $B_R^g(p)$ isometric to a Euclidean one, $u$ is radial and is given by 
	\begin{equation}\label{uesplbeta}
		u(x)=\frac{\beta}{2n}(R^2-d_g(x,p)^2)\, .
	\end{equation}
	\par Using \eqref{scalarprod} and \eqref{uesplbeta}, we first determine the form of the weight \(e^{-f}\). Then, by imposing condition \eqref{Ricf_hp}, we see that this leads to a contradiction. Thus, we conclude that the case $0<\beta<1$ cannot occour.
	
	Applying \Cref{farinaroncoroni} with \(k = 0\), from \eqref{formgOm} we obtain that the metric \(g\) has the form
	\begin{equation}\label{formgom1}
		g = dr \otimes dr + r^2 \, \overline g_{\alpha\beta}(\theta) \, d\theta^\alpha \otimes d\theta^\beta \, ,
	\end{equation}
	where \(\{r, \theta^\alpha\}_{\alpha=1}^{n-1}\) are coordinates as in \eqref{coordinates}, and \(\overline g_{\alpha\beta}(\theta)d\theta^\alpha\otimes d\theta^\beta\) is the standard metric on $S^{n-1}$ - the unit sphere in $T_pM$ - induced from $\mathbb{R}^n$. 
	Hence, the 1-form $df$ has the expression
	\begin{equation}\label{df}
		df = f_r \, dr + f_\alpha \, d\theta^\alpha \, ,
	\end{equation}
	where \(f_r\) and \(f_\alpha\) denote the partial derivatives with respect to \(r\) and \(\theta^\alpha\), respectively. Note that \eqref{scalarprod} can be rewritten as
	\begin{equation}\label{dfnablau}
		df(\nabla u) = 1 - \beta \, .
	\end{equation}
	Since \(u\) is radial, recalling that $r(x)=d_g(x,p)$, we have 
	$$
	\nabla u = u_r \, \nabla r \, .
	$$
	Combining this with \eqref{df} and \eqref{dfnablau}, we obtain
	$$
	f_r u_r = 1 - \beta \, .
	$$
	From \eqref{uesplbeta}, it follows that
	$$
	r f_r = - \frac{n(1-\beta)}{\beta} \, ,
	$$
	which integrates to
	\begin{equation}\label{weight}
		e^{-f} = r^{\frac{n(1-\beta)}{\beta}} \, e^{s(\theta)} \, ,
	\end{equation}
	for some smooth function $s(\theta)$.\\
	Next, we show that condition \eqref{Ricf_hp} cannot be satisfied in this case, that is, when $0<\beta<1$. Since \(\Omega\) is isometric to a Euclidean ball, we have
	$$
	\mathrm{Ric} = 0 \quad \text{in } \Omega \, .
	$$
	Hence, setting
	$$
	a \coloneq \frac{n(1-\beta)}{\beta} > 0 \, ,
	$$
	condition \eqref{Ricf_hp} reduces to
	\begin{equation}\label{cond_eucl}
		\nabla^2 f - \frac{1}{a} \, df \otimes df \geq 0 \quad \text{in } \Om\, .
	\end{equation}
	Recall that in local coordinates $\{x^i\}_{i=1}^n$, the Hessian of a smooth function $h$ is given by
	\begin{equation}\label{hessh}
		(\nabla^2 h)_{ij} = h_{ij} - \Gamma_{ij}^k \, h_k \, ,
	\end{equation}
	where $\Gamma_{ij}^k$ are the Christoffel symbols associated to the Levi-Civita connection of $g$:
	\begin{equation}\label{Christoffel}
		\Gamma_{ij}^k = \frac{1}{2} g^{kl} (\partial_i g_{jl} + \partial_j g_{il} - \partial_l g_{ij}) \, .
	\end{equation}
	From the form of $g$ in \eqref{formgom1}, it is straightforward to verify that the Christoffel symbols are given by
	$$
	\Gamma_{rr}^r = \Gamma_{rr}^\alpha = \Gamma_{r\alpha}^r = 0 \, , \quad 
	\Gamma_{r\alpha}^\beta = \frac{1}{r} \delta_\alpha^\beta \, , \quad 
	\Gamma_{\alpha\beta}^r = -r \, \overline g_{\alpha\beta} \, , \quad 
	\Gamma_{\alpha\beta}^\gamma = \hat \Gamma_{\alpha\beta}^\gamma \, ,
	$$
	where $\hat \Gamma_{\alpha\beta}^\gamma$ are the Christoffel symbols for $\overline g$.
	Hence, from \eqref{hessh} we get
	\begin{multline}\label{hessf}
		\nabla^2 f = f_{rr} \, dr \otimes dr + \left(f_{r\alpha} - \frac{1}{r} f_\alpha\right) (dr \otimes d\theta^\alpha + d\theta^\alpha \otimes dr) \\+ (f_{\alpha\beta} + r f_r \, \overline g_{\alpha\beta} - \hat\Gamma_{\alpha\beta}^\gamma f_\gamma) \, d\theta^\alpha \otimes d\theta^\beta \, .
	\end{multline}
	Moreover, we have
	\begin{equation}\label{dfotdf}
		df \otimes df = f_r^2 \, dr \otimes dr + f_r f_\alpha \, (dr \otimes d\theta^\alpha + d\theta^\alpha \otimes dr) + f_\alpha f_\beta \, d\theta^\alpha \otimes d\theta^\beta \, .
	\end{equation}
	Combining \eqref{hessf} and \eqref{dfotdf}, we obtain
	\begin{multline}\label{hessf-1adfotdf}
		\nabla^2 f - \frac{1}{a}\,  df \otimes df = \Big(f_{rr} - \frac{1}{a} f_r^2\Big) \, dr \otimes dr - \Big(-f_{r\alpha}+\frac{1}{r} f_\alpha + \frac{1}{a} f_r f_\alpha \Big) \, (dr \otimes d\theta^\alpha + d\theta^\alpha \otimes dr) \\  +\left(f_{\alpha\beta} + r f_r \, \overline g_{\alpha\beta} - \hat\Gamma_{\alpha\beta}^\gamma f_\gamma-\frac 1 a f_\al f_\bb\right) \, d\theta^\alpha \otimes d\theta^\beta \, .
	\end{multline}
	From \eqref{weight}, we have
	\begin{equation}\label{frfa}
		f_r = -\frac{a}{r} \, , \quad f_\alpha = -s_\alpha \, .
	\end{equation}
	Thus, substituting \eqref{frfa} into \eqref{hessf-1adfotdf}, condition \eqref{cond_eucl} becomes 
	\begin{equation}\label{condsn-1}
		\left(- a \, \overline g_{\alpha\beta} - s_{\alpha\beta} + \hat \Gamma_{\alpha\beta}^\gamma s_\gamma - \frac{1}{a} s_\alpha s_\beta \right) \, d\theta^\alpha \otimes d\theta^\beta \geq 0 \, .
	\end{equation}
	Note that 
	$$(\nabla^2_\theta s)_{\al\bb}=s_{\alpha\beta} - \hat \Gamma_{\alpha\beta}^\gamma s_\gamma \, ,$$
	where $\nabla^2_\theta s$ denotes the Hessian of $s$ on $(S^{n-1},\overline g)$. Thus, taking the trace of \eqref{condsn-1} on $S^{n-1}$ gives 
	$$a(n-1)+\lap_{\overline g} s +\frac 1 a |\nabla_\theta s|^2 \leq 0 \quad \text{on } S^{n-1} \, .$$
	Integrating over $S^{n-1}$ and using
	$$
	\int_{{S}^{n-1}} \Delta_{\overline g} s \, d{V_{\overline g}} = 0\, ,
	$$
	we obtain
	$$
	a(n-1)|{S}^{n-1}| + \frac{1}{a} \int_{{S}^{n-1}} |\nabla_\theta s|^2 \, d\tilde{V_{\overline g}} \leq 0\, ,
	$$
	which is a contradiction since $a>0$. Hence, the case $0<\beta<1$ cannot occur.\smallskip\\
	\indent \textbf{Case} $\bm{\beta=1:}$ Note that, in this case, \eqref{princ} becomes 
	\begin{equation}\label{betaeq1a}
		0\geq \int_\Omega u \, \Big\{ 
		\left|\nabla^2 u - \frac{\Delta_g u}{n} g \right|_g^2 
		+ \frac{2}{n} g(\nabla f, \nabla u)\lapw u \\
		+ \frac{1}{n} g(\nabla f, \nabla u)^2 
		+ (\mathrm{Ric}_g + \nabla^2 f)[\nabla u, \nabla u] 
		\Big\} \, d\mu_g \, ,
	\end{equation}
	If we applied the weighted Young inequality \eqref{youngw} as in the case $\beta<1$, \eqref{k} would yield $k=0$, not admissible since $k$ has to be positive. However, recalling that $\lapw u=-1$, we rewrite \eqref{betaeq1a} as
	\begin{equation}\label{betaeq1}
		0\geq \int_\Omega u \, \Big\{ 
		\left|\nabla^2 u - \frac{\Delta_g u}{n} g \right|_g^2 
		+ \frac{1}{n}g(\nabla f, \nabla u)(g(\nabla f,\nabla u)-2) 
		+ (\mathrm{Ric}_g + \nabla^2 f)[\nabla u, \nabla u] 
		\Big\} \, d\mu_g \, .
	\end{equation}
	From the assumption \eqref{iii} we have 
	\begin{equation}\label{condiii}
		g(\nabla f,\nabla u)(g(\nabla f,\nabla u)-2)\geq 0 \quad \text{in} \ \Om \,.
	\end{equation}
	We notice that \eqref{condiii} is satisfied also if
	\begin{equation*}
		g(\nabla f,\nabla u)\geq 2 \quad \text{in} \ \Om \, ,
	\end{equation*}
	which is not admissible, since in this case
	$$\lap_g u = \lapw u+g(\nabla f,\nabla u)\geq 1 \, .$$ 
	and the maximum principle implies $u<0$ in $\Om$, a contradiction.\\ 
	Hence, from
	\begin{equation}\label{ricinf}
		(\ric_{f}^{\infty})[\nabla u,\nabla u] = (\ric_g+\nabla^2 f)[\nabla u,\nabla u] \geq 0 \, ,
	\end{equation}
	using \eqref{inhessu}, \eqref{condiii}, \eqref{ricinf}, and the fact that $u$ is positive in $\Om$, we get that the equality sign is attained in \eqref{betaeq1} and, moreover, equality is attained in all the inequalities \eqref{iii}, \eqref{hessV}, \eqref{inhessu} and \eqref{ricinf}. Now, as in the case $\beta<1$, by applying \Cref{farinaroncoroni} to $u$ with $k=0$, we conclude that $\Om$ is a metric ball isometric to a Euclidean one, $u$ is radial and is explicitly given by \eqref{uespl}.
	\par Finally, we show that, in this case, condition \eqref{Ricf_hp} forces the weight $e^{-f}$ to be constant.\\
	Note that equality in \eqref{iii} correponds to
	\begin{equation}\label{iiieq}
		g(\nabla f, \nabla u) = 0 \quad \text{in } \Omega\, ,
	\end{equation}
	which coincides with \eqref{scalarprod} for $\beta = 1 $. Therefore, proceeding as in the case $0 < \beta < 1 $, we obtain \eqref{weight}, which now reads
	\begin{equation}\label{weightbeta1}
		e^{-f} = e^{s(\theta)}
	\end{equation}
	for some smooth function $s(\theta)$.
	Moreover, in this setting, condition \eqref{Ricf_hp} becomes
	$$
	\ric + \nabla^2 f \geq 0 \quad \text{in } \Omega\, .
	$$
	Since $\Omega$ is isometric to a Euclidean ball, we have $\ric = 0$, and hence
	$$
	\nabla^2 f \geq 0 \quad \text{in } \Om\, ,
	$$
	which in particular implies
	\begin{equation}\label{lapfgeq0}
		\Delta_g f \geq 0 \quad \text{in } \Om\, .
	\end{equation}
	Applying the divergence theorem in $\Omega$, and using \eqref{iiieq} together with the boundary condition $u=0$ on $\p\Om$, we obtain
	$$
	\int_\Omega u \, \Delta_g f \, dV_g= 0 \, .
	$$
	Hence, since $u > 0$ in $\Om$, from \eqref{lapfgeq0} it follows that
	$$
	\Delta_g f \equiv 0 \quad \text{in } \Om \, .
	$$
	Consequently, using \eqref{weightbeta1}, we deduce that
	$$
	\Delta_{\overline g} s \equiv 0 \quad \text{on } S^{n-1}.
	$$
	Since $(S^{n-1}, \overline{g})$ is a compact Riemannian manifold without boundary, this implies that $s$ is constant. Therefore, by \eqref{weightbeta1}, the weight $e^{-f}$ must also be constant.
\end{proof}
\begin{proof}[An alternative proof of \Cref{main} (via the P-function method)]
	Let 
	\begin{equation}\label{Pfunction}
		P(u)\coloneq |\nabla u|_g^2+\frac{2\beta}{n} u \, .
	\end{equation}
	We distinguish between the cases $0<\beta<1$ and $\beta =1$.\\
	\textbf{Case} $\bm{0<\beta <1:}$
	From \Cref{weighted_Bochner} with $m=\frac n \beta>n$, we have that $\lapw P(u)\geq 0$, since by assumption $\ric_f^m\geq 0$. Then, from the strong maximum principle, $P(u)$ cannot attain a maximum in $\Om$ unless it is constant throughout $\Om$. Hence, since $\overline\Om^g$ is compact, $P(u)\in C^0(\overline\Om^g)$ and $P(u)=c^2$ on $\partial\Om$, then either 
	\begin{equation}\label{Pconst}
		P(u)\equiv c^2 \quad \text{in} \ \Om \, ,
	\end{equation}
	or
	\begin{equation}\label{Pleq}
		P(u)<c^2 \quad \text{in} \ \Om \, .
	\end{equation}
	By contradiction assume that \eqref{Pleq} is satisfied. Thus,
	\begin{equation*}
		|\nabla u|_g^2 +\frac {2\beta}{n}u <c^2 \quad \text{in} \ \Om\, .
	\end{equation*}
	By integrating in $\Om$ 
	\begin{equation}
		\int_\Om |\nabla u|_g^2 \, d\mu_g +\frac{2\beta}{n}\int_\Om u \, d\mu_g< c^2 |\Om|_f \, ,
	\end{equation}
	and using the Pohozaev identity \eqref{pohozaev}, we obtain
	$$\int_\Om |\nabla u|_g^2 \, d\mu_g +\frac{2\beta}{n}\int_\Om u \, d\mu_g<\int_\Om u \, d\mu_g +2 \int_\Om \nabla^2 V(\nabla u,\nabla u) \, d\mu_g \, .$$
	Assumption
	$$\nabla ^2 V\leq \frac \beta n g $$
	and \eqref{intnablau} yield
	$$\int_\Om |\nabla u|_g^2 \, d\mu_g +\frac{2\beta}{n}\int_\Om u \, d\mu_g<\int_\Om |\nabla u|_g^2\, d\mu_g +\frac{2\beta}{n}\int_\Om u \, d\mu_g \, ,$$
	a contradiction. Thus, $P(u)$ must be constant and, consequently, $\lapw P(u)=0$. In particular, equality is attained in \eqref{weighted_Bochnineq} and from \Cref{weighted_Bochner} we have
	$$\nabla^2 u =\frac{\lap_g u}{n} g \quad \text{in} \ \Om$$
	and 
	$$\lapw u =-\frac{m}{m-n}\, g(\nabla f,\nabla u) \quad \text{in} \ \Om\, .$$
	Since $m=\frac n \beta$ and $\lapw u=-1$, we get
	$$\nabla^2 u=-\frac \beta n g \quad \text{in} \ \Om$$
	and 
	$$g(\nabla f,\nabla u)=1-\beta \quad \text{in} \ \Om \, ,$$
	which correspond to \eqref{hessuespl} and \eqref{scalarprod}, respectively. Now, we obtain a contradiction by arguing as in the previous proof. Hence, the case $0<\beta<1$ cannot occour.\\
	\textbf{Case} $\bm{\beta =1:}$ In this case, we can not apply \Cref{weighted_Bochner} as above, since $m=n$. However, arguing as in the proof of \Cref{weighted_Bochner}, we have that
	\begin{align*}
		\frac 1 2 \lapw P(u)&\geq \frac 1 n (\lapw u+g(\nabla f,\nabla u))^2+\ric_f^\infty(\nabla u,\nabla u) +\frac 1 n \lapw u \\
		& = -\frac 2 n g(\nabla f,\nabla u)+\frac 1 n g(\nabla f,\nabla u)^2+\ric_f^\infty(\nabla u,\nabla u)\\
		&=\frac 1 n g(\nabla f,\nabla u)(g(\nabla f,\nabla u)-2)+\ric_f^\infty(\nabla u,\nabla u) \, ,
	\end{align*}
	where we have used the Bochner-Weitzenb\"{o}ck formula \eqref{bochnweitz}, \eqref{cs} and $\lapw u=-1$ in $\Om$.
	Since by assumption 
	\begin{equation}\label{assiii}
		g(\nabla f,\nabla u)\leq 0 \quad \text{in}\ \Om
	\end{equation}
	and $\ric_f^\infty\geq 0$, it follows that
	$$\lapw P(u)\geq 0 \quad \text{in} \ \Om \, .$$
	As in the case $\beta<1$, integrating \eqref{Pleq} and applying the Pohozaev identity \eqref{pohozaev} leads to a contradiction. Therefore, by the maximum principle, 
	$$\lapw P(u)\equiv 0\quad \text{in} \ \Om \, .$$ 
	Thus, equality is attained both in \eqref{cs} and \eqref{assiii}, which gives
	$$\nabla^2 u =\frac{\lap_g u}{n}g$$ 
	and
	$$g(\nabla f,\nabla u)=0 \quad \text{in} \ \Om \, ,$$
	respectively. 
	Recalling that $\lapw u=\lap_g u-g(\nabla f,\nabla u)$, we obtain
	$$\nabla^2 u =-\frac 1 n g \, .$$
	Hence, the application of \Cref{farinaroncoroni} completes the proof.
\end{proof}
	We now compare the two proofs of \Cref{main} given above.\\
	Recall that we define
	$$P(u)\coloneq |\nabla u|_g^2+\frac{2\beta}{n}u \, .$$
	In the $P$-function approach, we aim to prove that equality holds in \eqref{bochneq} and \eqref{bochneq1}. Then, using the equation $\lapw u=-1$ in $\Om$ and applying \Cref{farinaroncoroni}, one deduces that $\Om$ is isometric to a Euclidean ball, $u$ is radial and its explicit form is determined.\\
	From \Cref{weighted_Bochner} it follows that $\lapw P(u)\geq 0$ in $\Om$ and equality is attained in \eqref{bochneq} and \eqref{bochneq1} if and only if
	\begin{equation}\label{lapPzero}
		\lapw P(u)\equiv 0 \quad \text{in} \ \Om \, .
	\end{equation} 
	This condition is achieved by applying the strong maximum principle together with the Pohozaev identity \eqref{pohozaev}. Now, we show that our first argument - entirely based on integral identities - actually establishes \eqref{lapPzero} without invoking the strong maximum principle for $P$.\\
	Using the generalized Bochner-Weitzenb\"{o}ck formula \eqref{bochnweitz} we have
	\begin{equation}\label{lapP}
		\frac 1 2 \lapw P(u) 
		= |\nabla^2 u|_g^2 + \ric_f^\infty(\nabla u,\nabla u)+\frac \beta n \lapw u \, .
	\end{equation} 
	Thus, using \eqref{lapP}, we can rewrite \eqref{idciraolo} as
	\begin{equation}\label{idcP}
		\begin{aligned}
			&e^f\divg_g \left[e^{-f}\left(u\nabla \left(\frac{|\nabla u|_g^2}{2}\right)-u\,\lapw u \nabla u-\frac 1 2 |\nabla u|_g^2\nabla u\right)\right]\\
			&=  u \, \frac{\lapw P(u)}{2}-u (\lapw u)^2-\frac 3 2 |\nabla u|_g^2 \lapw u -u \, \frac \beta n \lapw u \, .
		\end{aligned}
	\end{equation}
	Therefore, by proceeding exactly as in the first proof of \Cref{main} — namely, integrating \eqref{idcP} in $\Om$, using the divergence theorem together with $\lapw u=-1$ in $\Om$ and applying the Pohozaev identity \eqref{pohozaev} — we get
	\begin{equation*}
		0= \int_\Om u \, \frac{\lapw P(u)}{2} \, d\mu_g -\int_\Om \nabla^2 V(\nabla u,\nabla u) \, d\mu_g +\frac \beta n \int_\Om u \, d\mu_g \, .
	\end{equation*}
	Hence, using \eqref{intnablau} and the assumption 
	$$\nabla^2 V\leq \frac \beta n g \, ,$$
	we find 
	\begin{equation}\label{intPin}
		0\geq  \int_\Om u \, \frac{\lapw P(u)}{2} \, d\mu_g \, .
	\end{equation}
	Next, combining \eqref{lapP} with \eqref{id}, we compute
	\begin{equation*}
		\frac 1 2 \lapw P(u) 
		\begin{multlined}[t]
			= (\lapw u)^2+\left|\nabla^2 u-\frac{\lap_g u}{n}g\right|_g^2-\frac{n-1}{n}(\lapw u)^2 +\frac 2 n g(\nabla f,\nabla u)\lapw u +\frac 1 n g(\nabla f,\nabla u)^2 \\+\ric_f^\infty(\nabla u,\nabla u)+\frac \beta n \lapw u \, .
		\end{multlined}
	\end{equation*}
	Since $\lapw u=-1$ in $\Om$, this simplifies to
	\begin{equation}\label{lapP1}
		\frac 1 2 \lapw P(u)=\left|\nabla^2 u-\frac{\lap_g u}{n}g\right|_g^2 +\frac{1-\beta}{n}+\frac 2 n g(\nabla f,\nabla u)\lapw u +\frac 1 n g(\nabla f,\nabla u)^2+\ric_f^\infty(\nabla u,\nabla u) \,.
	\end{equation}
	Hence, multiplying by $u$ and integrating in $\Om$, we deduce that
	\begin{equation}\label{intlapP}
		\begin{multlined}
			\int_\Om u \,\frac {\lapw P(u)}{2} \, d\mu_g
			=\left(\frac{1-\beta}{n}\right)\int_\Om u \, d\mu_g  \\+ \int_\Omega u \, \Big\{ 
			\left|\nabla^2 u - \frac{\Delta_g u}{n} g \right|_g^2 
			+ \frac{2}{n} g(\nabla f, \nabla u)\lapw u 
			+ \frac{1}{n} g(\nabla f, \nabla u)^2 
			+ \ric_f^\infty(\nabla u, \nabla u) 
			\Big\} \, d\mu_g \, .
		\end{multlined}
	\end{equation} 
	Our argument shows that, under the assumptions of \Cref{main}, the integrand on the right-hand side of \eqref{intlapP} - which coincides with $u \,\lapw P(u)/2$ from \eqref{lapP1} - is nonnegative. Therefore, since $u>0$ in $\Om$, it follows that
	\begin{equation}\label{lapPgeq0}
		\lapw P(u)\geq 0 \quad \text{in} \ \Om \, .
	\end{equation} 
	Combining \eqref{lapPgeq0} with \eqref{intPin}, we conclude that
	$$\int_\Om u \, \frac{\lapw P(u)}{2} = 0 \, .$$
	Since $u$ is positive, from \eqref{lapPgeq0} we deduce that
	$$\lapw P(u)\equiv 0 \quad \text{in} \ \Om \, ,$$
	and hence \eqref{lapPzero} holds.
\section{The non-compact case}\label{nccompact_case}
\subsection{Preliminary results}\label{preliminarynoncompl}
In this section, we establish several preliminary results that will be essential for the proof of \Cref{main_nc}. We begin by proving a version of the divergence theorem in the Riemannian setting $(\R^n\setminus\{O\},g)$, where the metric $g$ is defined as in \eqref{g}. We then derive some fundamental properties of solutions to \eqref{ourprob}, and conclude with a discussion on the energy space $W^{1,2}_0(\Om,d\mu_g)$ where the problem is formulated. \\Throughout, we consider $\Om$ to be a set in $(\mathbb{R}^n \setminus \{O\}, g)$ satisfying the following assumptions:
\begin{equation} \label{domain}
	\begin{split}
		\Omega \text{ is a bounded, open, and connected set in } (\mathbb{R}^n \setminus \{O\}, g) \\
		\text{ with compact boundary } \partial\Omega \text{ of class } C^{1,s}.
	\end{split}
\end{equation}
 We denote by $\overline \Om$ the closure of $\Om$ with respect to the Euclidean metric in $\R^n$. Moreover, we consider a larger range of parameters than the one in \Cref{main_nc} and, for simplicity, we assume only $\gamma < 1$ a priori; as shown in \Cref{gammamin1}, this assumption is not restrictive. It ensures that the origin does not lie at infinity. Indeed, when $\gamma>0$, the origin of $\R^n$ represents a singularity, since the metric $g$ is not defined at this point. Now, we show that, if $\gamma <1$, then this singularity lies at a finite distance. More precisely, for a fixed point $x \in \R^n\setminus\{O\}$, consider
\begin{equation}\label{defdistance}
	d_g(x,O)\coloneq \inf\{L(\sigma) : \sigma:[0,1)\to M \ \text{admissible curve}, \sigma(0)=x, \lim_{t \to 1^-} \sigma(t) = O\} \, ,
\end{equation}
where $L(\sigma)$ denotes the length of the curve $\sigma$ with respect to the metric $g$.
\begin{lemma}\label{Ofinitedist}
	Let $\gamma<1$. Then,
	\begin{equation}\label{dist0}
		d_g(x,O)=\frac{|x|^{1-\gamma}}{1-\gamma} \, , \quad x \in \R^n\setminus\{O\} 
	\end{equation}
	and
	\begin{equation}\label{ballgeucl}
		B_R^g = B_{((1-\gamma)R)^{\frac{1}{1-\gamma}}}^\delta(O)\setminus\{O\}  \, .
	\end{equation}
	Let $\Om$ satisfy \eqref{domain}. If $\alpha>-n$, then 
	\begin{equation*}
		e^{-f} \in L^1(\Om,dV_g) \, ,
	\end{equation*}
	where $e^{-f}$ is defined by \eqref{pesopower}.
    In this case, the weighted volume of $B_R^g$ is given by
	\begin{equation}\label{wvolumeball}
		|B_R^g|_f =w_n (1-\gamma)^{n-1} \frac{R^{\alpha+n}}{\alpha+n} \, ,
	\end{equation}
	where $w_n=|\Sp|$ is the measure of the unit sphere.
\end{lemma}
\begin{proof}
	Let $\sigma$ be a curve as in \eqref{defdistance}. Write $\sigma$ in polar coordinates as
	$$\sigma(t)=r(t)\theta(t) \, ,$$
	where $r(t)=|\sigma(t)|$ and $\theta(t)=\sigma(t)/|\sigma(t)| \in \Sp$.
	We have
	\begin{equation}\label{dersigma}
		|\dot{\sigma}|^2 = |\dot r |^2+r^2 |\dot \theta|^2 \, .
	\end{equation}
	Hence, the lenght of $\sigma$ satisfies
	$$L(\sigma)=\int_0^1 \sqrt{g(\dot \sigma,\dot \sigma)} \, dt\\
	=\int_0^1 |\sigma(t)|^{-\gamma} |\dot \sigma(t)| \, dt \, .$$
	Since from \eqref{dersigma} it follows that $|\dot \sigma |\geq |\dot r|$, we obtain
	$$L(\sigma)\geq \int_0^1 r(t)^{-\gamma} |\dot r(t)| \, dt \geq \left|\int_{|x|}^{0} r^{-\gamma} \, dr \right| \, .$$
	If $\gamma<1$ the last integral converges, giving
	\begin{equation}\label{distin}
		d_g(x,O)\geq \frac{|x|^{1-\gamma}}{1-\gamma} \, .
	\end{equation}
	Now, consider the radial curve
	$$\tilde\sigma(t)\coloneq (1-t)x \, .$$
	which lies in the class given by \eqref{defdistance}. It is straightforward to see that 
	$$L(\tilde\sigma)=\frac{|x|^{1-\gamma}}{1-\gamma} \, .$$
	Thus, from \eqref{distin} we conclude \eqref{dist0}.\\ Note that this also shows that $(M,g)$ is geodesically incomplete: the radial curve $\tilde\sigma$ is a geodesic, as $L(\tilde\sigma) = d_g(x,O)$, and it reaches $p=O$ in finite time. Since the metric is not defined at the singularity, $\tilde\sigma$ cannot be extended beyond $t=1$.
	\par Let $\Om$ satisfy \eqref{domain}. Recall that 
	$$e^ {-f}=d_g(x,O)^\alpha \, .$$
	From \eqref{dist0}, this can be rewritten as
	$$e^{-f}=\frac{|x|^{\alpha(1-\gamma)}}{(1-\gamma)^\alpha} \, .$$
	Note that, if the origin does not belong to $\overline\Om$, then $e^{-f} \in C^0(\overline\Om)$, and hence $e^{-f}\in L^1(\Om,dV_g)$. Thus, we may assume that $O\in \overline \Om$. In this case, $e^{-f} \in L^1(\Om,dV_g)$ if and only if $e^{-f} \in L^1(B_R^g,dV_g)$, for any $R>0$. Let $w_n=|\Sp|$ denote the measure of the unit sphere.
	Using \eqref{ballgeucl}, we compute
	\begin{equation}\label{measwBR}
		\int_{B_R^g} e^{-f} \, dV_g = \frac{w_n}{(1-\gamma)^\alpha} \int_0^{((1-\gamma)R)^{1/(1-\gamma)}} r^{(1-\gamma)(\alpha+n)-1} \, dr \, ,
	\end{equation}
	which is finite if and only if $\alpha>-n$.
	\par Finally, note that, if $\alpha>-n$, from \eqref{measwBR} we get
	$$|B_R^g|_f = w_n (1-\gamma)^{n-1} \frac{R^{\alpha+n}}{\alpha+n} \, ,$$
	which concludes the proof.
\end{proof}
\begin{remark}
	In this setting we recall that for conformal metrics $g$ and $\delta$ we have (see for instance \cite[Formulas (2.68), (2.74), (2.77) and (2.94)]{CatinoMastrolia})
	\begin{align}\label{ric}
		\ric_g &= \ric_\delta -(n-2)(D^2\phi -d\phi\otimes d\phi) -(\lap\phi+(n-2)|D\phi|^2)\delta\nonumber\\
		&=-(n-2)(D^2\phi -d\phi\otimes d\phi) -(\lap\phi+(n-2)|D\phi|^2)\delta \, ,
	\end{align}
	since $\ric_\delta=0$. If $u$ is a smooth function and $X$ is a smooth vector field on $M$, then the following holds:
	\begin{align}
		&\nabla u = e^{-2\phi}Du \label{grad} \, ,\\
		&\nabla^2 u = D^2u-(du\otimes d\phi+d\phi\otimes du)+(Du\cdot D\phi)\delta \label{hess} \, ,\\
		&\lap_g u = e^{-2\phi}(\lap u+(n-2)(Du\cdot D\phi)) \label{lap} \, ,\\
		&\divg_g X = \divg X+n X \cdot D\phi\label{div}\, .
	\end{align}
\end{remark}
\par Note that, since the origin is excluded from the manifold, $\overline\Om^g$  is generally not compact in $(\R^n\setminus \{O\},g)$. As a consequence, the divergence theorem cannot be applied directly on $\Om$ for all vector fields that are smooth on $\overline \Om^g$. To address this, we prove a useful lemma based on a suitable cutoff function argument, which we will apply to justify the integration by parts in what follows.
\begin{lemma}\label{lemmadivthm}
	Let $\Om$ satisfy \eqref{domain} and let $g$ be given by \eqref{g}. Let $F \in C^0(\overline\Om^g)$ be a vector field such that $|F|_g \in L^2(\Om,d\mu_g)$, and either $\divg_g F \in L^1(\Om,d\mu_g)$ or $\divg_g (e^{-f}F)\geq 0$ in $\Om$. If
	\begin{equation}\label{condgammadiv}
		\gamma<1 \quad \text{and} \quad \alpha>2-n \, ,
	\end{equation}
	then
	\begin{equation}
		\int_\Om \divg_g(e^{-f}F)\, dV_g= \int_{\p\Om}e^{-f}g(F,\nu) \, d\tilde V_g \,  .
	\end{equation}
\end{lemma}
\begin{remark}\label{lemmadivutile}
	If the origin does not belong to $\overline \Om \subset \R^n$, then $\overline \Om^g \subset \R^n\setminus\{O\}$ is compact in $(\R^n\setminus\{O\},g)$ and the divergence theorem can be applied for all $F \in C^1(\overline \Om^g)$. Thus, this lemma is useful only when the origin in fact belongs to $\overline \Om$.
\end{remark}
\begin{remark}
	In the proof, we first multiply the vector field $F$ by a cutoff function with compact support in $\Omega$, allowing the classical divergence theorem to be applied in a localized setting. Although it is usually stated under the stronger regularity assumption $F \in C^1(\overline{\Omega}^g)$, the theorem remains valid under weaker conditions. In particular, it suffices to assume $F \in C^0(\overline{\Omega}^g)$; see \cite[Lemma 4.3]{CiraoloLi}.
\end{remark}
\begin{proof}
	According to \textit{Remark} \ref{lemmadivutile}, we assume that $0 \in \overline\Om$. As already recalled, we define
	\begin{equation*}
		B_R^g \coloneqq \{x \in \R^n\setminus\{O\} : d_g(x,O)< R\} \, .
	\end{equation*} 
	Let $\eps>0$ be such that $B_{4\eps}^g \subset \Om$, $\varphi_\eps \in C^\infty_c(\Om)$ a cut-off function such that $\varphi_\eps \equiv 0$ in $B_{\eps}^g$, $\varphi_\eps \equiv 1$ in $\Om\setminus B_{2\eps}^g$, $0\leq \varphi_\eps \leq 1$, and $|\nabla\varphi_\eps|_g\leq \frac 2 \eps$. \\Since the divergence theorem applies in $\Om\setminus\overline B_\eps^g$, we have 
	\begin{equation}\label{reldiv}
		\int_\Om \divg_g(e^{-f}F\pphi_\eps) \, dV_g =\int_{\Om\setminus B_\eps^g} \divg_g(e^{-f}F\pphi_\eps) \, dV_g =\int_{\p\Om} e^{-f} g(F,\nu)\, d\tilde V_g \, .
	\end{equation}
	Thus, it suffices to prove that
	\begin{equation}\label{limintdivg}
		\int_\Om \divg_g(e^{-f}F\pphi_\eps) \, dV_g \to \int_\Om \divg_g(e^{-f}F) \, dV_g \quad \text{as} \ \eps\to 0 \, .
	\end{equation}
	Note that 
	$$\divg_g(e^{-f}F\pphi_\eps)= e^{-f}\left\{g(F,\nabla \pphi_\eps)+\left(\divg_gF-g(\nabla f,F)\right)\pphi_\eps\right\} \quad \text{in} \ \Om \, .$$
	Hence, recalling that $d\mu_g = e^{-f}dV_g$, we get
	\begin{equation}\label{intdivg}
		\int_{\Om\setminus B_\eps^g} \divg_g(e^{-f}F \pphi_\eps) \, dV_g = \int_{\Om\setminus B_\eps^g} g(F,\nabla \pphi_\eps) \, d\mu_g +\int_{\Om\setminus B_\eps^g} (\divg_g F-g(\nabla f,F)) \, \pphi_\eps \, d\mu_g \, .
	\end{equation}
	Thus, to establish \eqref{limintdivg}, it is enough to show that
	\begin{equation}\label{limto0}
		\int_{\Om\setminus B_\eps^g} g(F,\nabla \pphi_\eps) \, d\mu_g \to 0  \quad \text{as} \ \eps \to 0  \, ,
	\end{equation}
	and
	\begin{equation}\label{limtodivg}
		\begin{multlined}
			\int_{\Om\setminus B_\eps^g} (\divg_g F-g(\nabla f,F))\, \pphi_\eps \, d\mu_g \to \int_{\Om}(\divg_g F-g(\nabla f,F)) \, d\mu_g =\int_\Om \divg_g(e^{-f}F)\, dV_g \\\text{as } \eps\to 0 \, .
		\end{multlined}
	\end{equation}
	We first verify \eqref{limto0}. Since $\nabla \pphi_\eps$ is supported in $B_{2\eps}^g\setminus B_\eps^g$, $|F|_g \in L^2(\Om,d\mu_g)$ and $|\nabla \pphi_\eps|_g\leq \frac 2 \eps$, we have
	\begin{align*}
		\int_{\Om\setminus B_\eps^g} g(F,\nabla \pphi_\eps) \, d\mu_g&=\int_{\Om} g(F,\nabla \pphi_\eps)\chi_{\Om\setminus B_\eps^g} \, d\mu_g \\
		&\leq \left(\int_{\Om} |F|_g^2 \, d\mu_g\right)^{1/2} \left(\int_\Om|\nabla \pphi_\eps|_g^2 \, d\mu_g \right)^{1/2}\\
		&\leq 4\left(\int_{\Om} |F|_g^2 \, d\mu_g\right)^{1/2} \left(\frac{1}{\eps^2}|B_{2\eps}^g\setminus B_{\eps}^g|_f\right)^{1/2} \, .
	\end{align*}
	Using \eqref{wvolumeball}, we obtain
	\begin{equation}\label{measdiffball}
		\frac{1}{\eps^2}|B_{2\eps}^g\setminus B_{\eps}^g|_f\sim \eps^{\alpha+n-2} \to 0 \quad \text{as} \ \eps \to 0\, ,
	\end{equation} 
	since by assumption $\alpha>2-n$.
	Consequently, \eqref{limto0} holds.\\
	\indent To prove \eqref{limtodivg}, note that:
	\begin{equation}\label{secdiv}
		\int_{\Om\setminus B_\eps^g} (\divg_g F-g(\nabla f,F))\, \pphi_\eps \, d\mu_g=\int_{\Om} (\divg_g F-g(\nabla f,F)) \, \pphi_\eps \, \chi_{\Om\setminus B_\eps^g} \, d\mu_g \, .
	\end{equation}
	\textbf{Case 1}: $\divg_g F \in L^1(\Om,d\mu_g)$. Recalling that $\pphi_\eps\leq 1$, we have
	$$|(\divg_g F-g(\nabla f,F)) \, \pphi_\eps \, \chi_{\Om\setminus B_\eps^g}|\leq |\divg_g F-g(\nabla f,F)|\leq |\divg_g F|+|\nabla f|_g |F|_g \, .$$
	If we show that $|\divg_gF|+|\nabla f|_g|F|_g \in L^1(\Om,d\mu_g)$, then we can apply the dominated convergence theorem to \eqref{secdiv} and obtain the desired result. By assumption, $\divg_g F \in L^1(\Om,d\mu_g)$ and $|F|_g \in L^2(\Om,d\mu_g)$. Thus, it sufficies to show that $|\nabla f|_g \in L^2(\Om,d\mu_g)$. Note that 
	$$f(x)=-\alpha(1-\gamma) \log|x|-\alpha \log (1-\gamma) \, ,$$
	since, from \eqref{dist0} we have $$e^{-f}=\frac{|x|^{\alpha(1-\gamma)}}{(1-\gamma)^\alpha} \, .$$
	Recalling that $g_{ij}=|x|^{-2\gamma}\delta_{ij}$ and using \eqref{grad}, we compute
	$$|\nabla f|_g^2=g(\nabla f,\nabla f)=|x|^{2\gamma}|Df|^2 =  \frac{\alpha^2(1-\gamma)^2}{|x|^{2-2\gamma}} \, .$$
	Therefore,
	\begin{equation}\label{nablafinl2}
		\int_\Om|\nabla f|_g^2 \, d\mu_g =\int_\Om \frac{\alpha^2(1-\gamma)^2}{|x|^{2-2\gamma}} \, d\mu_g =\frac{\alpha^2(1-\gamma)^2}{(1-\gamma)^\alpha}\int_\Om |x|^{2(\gamma-1)+\alpha(1-\gamma)-n\gamma} \, dx \, ,
	\end{equation}
	which is finite under the assumption \eqref{condgammadiv}, recalling that, from \eqref{ballgeucl}, $\Om$ is bounded with respect to the Euclidean metric.\\
	\textbf{Case 2}: $\divg_g (e^{-f}F)\geq 0$ in $\Om$. Here, we apply the monotone convergence theorem to conclude
	$$\int_\Om \divg_g (e^{-f}F) \, \pphi_\eps\, d\mu_g \to \int_\Om \divg_g (e^{-f}F) \, d\mu_g \quad \text{as} \ \eps \to 0  \, ,$$
	and thus \eqref{limtodivg} holds in this case as well.
\end{proof}
We now establish two useful properties of solutions to the problem \eqref{ourprob}.
\begin{lemma}\label{nonneg}
	Let $\Om$ be a domain satisfying \eqref{domain}. If $u \in W_0^{1,2}(\Om,d\mu_g)$ is a solution to \eqref{ourprob}, then $u$ is positive in $\Om$.
\end{lemma}
\begin{proof}
	Recall that $u \in C^1(\overline\Om^g)$. We first show that $u$ is nonnegative in $\Om$ by testing the equation in \eqref{ourprob} against its negative part. Define $$v\coloneqq \min\{u,0\} \quad \text{in}  \ \Om \, .$$ 
	It is straightforward to verify that $v \in W_0^{1,2}(\Om,d\mu_g) \, .$ Taking $\pphi=v$ as a test function in the weak formulation \eqref{weak}, we obtain
	$$\int_\Om g(\nabla u,\nabla v) \, d\mu_g = \int_\Om v \, d\mu_g \, ,$$
	which reduces to
	$$\int_{\{u<0\}} |\nabla u|_g^2 \, d\mu_g =\int_{\{u<0\}} u \, d\mu_g \, .$$
	Note that the left-hand side is nonnegative. If the measure of \(\{u < 0\}\) were positive, then the right-hand side would be strictly negative, leading to a contradiction. Therefore, the measure of \(\{u < 0\}\) must be zero. Hence, $u\geq 0$ almost everywhere in $\Om$ and since $u \in C^1(\overline\Om^g)$, we conclude that $u\geq 0$ in $\Om$. To prove the positivity of $u$, suppose by contradiction that $u(x_0)=0$ for some $x_0 \in \Om$. Thus, $x_0$ is a minimum point for $u$ and $Du(x_0)=0$, $D^2u(x_0)\geq 0$. Hence, recalling \eqref{lap} and \eqref{grad}, we have
	\begin{align*}
		-1&=\lapw u(x_0)\\
		&=e^{-2\phi(x_0)}(\lap u(x_0)+(n-2)Du(x_0)\cdot D\phi(x_0)-Df(x_0)\cdot Du(x_0))\\
		&=e^{-2\phi(x_0)}\lap u(x_0)\geq 0 \, ,
	\end{align*}
	a contradiction. 
\end{proof}
In the statements that follow, we shall assume
\begin{equation}\label{condgammadivpoho}
	\gamma<1 \quad \text{and} \quad \alpha>2-n \, ,
\end{equation}
in order to apply \Cref{lemmadivthm} and to ensure that the weighted volume of $\Om$ is finite, i.e., $e^{-f} \in L^1(\Om,dV_g)$, as observed in \Cref{Ofinitedist}.\\
We first establish \Cref{propsolutions} in this framework, following the same argument and invoking \Cref{lemmadivthm} to justify the use of the divergence theorem.
\begin{lemma}
	Let $\Om$ be a domain satisfying \eqref{domain}, and let $u \in W^{1,2}_0(\Om,d\mu_g)\cap L^\infty(\Om)$ be a solution to \eqref{ourprob}. Assume that $\alpha$ and $\gamma$ satisfy \eqref{condgammadivpoho}.
	Then it holds
	\begin{equation}\label{cweight_nc}
		c=\frac {|\Om|_f}{|\p \Om|_f} \, ,
	\end{equation}
	and
	\begin{equation}\label{intnablau_nc}
		\int_\Om |\nabla u|_g^2 \,d\mu_g = \int_\Om u \,d\mu_g \, .
	\end{equation}
\end{lemma}
\begin{proof}
	Define $$F\coloneqq\nabla u \, ,$$ so that $|F|_g \in L^2(\Om,d\mu_g)\cap C^0(\overline \Om^g)$, since $u \in W_0^{1,2}(\Om,d\mu_g)\subset W^{1,2}(\Om,d\mu_g)$ and \\$u \in C^1(\overline \Om^g)$. Moreover, from the equation in \eqref{ourprob}, we have
	$$|\divg_g(F)|=|\lap_g u|=|g(\nabla f,\nabla u)-1|\leq |\nabla f|_g |\nabla u|_g+1 \in L^1(\Om,d\mu_g) \, ,$$
	since $|\nabla u|_g=|F|_g \in L^2(\Om,d\mu_g)$ and $|\nabla f|_g \in L^2(\Om,d\mu_g)$ under assumption \eqref{condgammadivpoho}, as established in \eqref{nablafinl2}. Therefore, \Cref{lemmadivthm} applies and from $\lapw u=-1$ we obtain
	$$|\Om|_f=\int_\Om e^{-f}\,dV_g=-\int_\Om \divg_g(e^{-f}\nabla u)\,dV_g=-\int_{\p\Om} e^{-f}g(\nabla u,\nu)\,d\tilde V_g = c\,|\partial\Om|_f \, ,$$
	from which we deduce \eqref{cweight_nc}. Moreover, we have
	\begin{align*}
		\int_\Om u e^{-f}\,dV_g &=-\int_\Om u e^{-f} \lapw u \,dV_g\\
		&=-\int_\Om u\divg_g(e^{-f}\nabla u)\,dV_g\\
		&=\int_\Om g(\nabla u,\nabla u)e^{-f}\,dV_g-\int_{\p\Om}ue^{-f}g(\nabla u,\nu)\,d\tilde V_g \, ,
	\end{align*}
	from which we conclude \eqref{intnablau_nc}, since $u=0$ on $\p\Om$.
	The last equality follows from a cutoff argument similar to that used in the proof of \Cref{lemmadivthm}, relying on the fact that $u \in L^\infty(\Om)$. For completeness, we report all the details. 
	Let $\eps>0$ be such that $B_{4\eps}^g$ is contained in $\Om$, where $B_{4\eps}^g$ is defined as in \eqref{defball}. Consider a smooth cutoff function \(\varphi_\eps \in C_c^\infty(\Omega)\) such that \(\varphi_\eps \equiv 0\) on \(B_\varepsilon^g\), \(\varphi_\eps\equiv 1\) on \(\Omega \setminus B_{2\varepsilon}^g\), with \(0 \leq \varphi_\eps \leq 1\) and \(|\nabla \varphi_\eps|_g \leq \frac{2}{\varepsilon}\). As before, set $F=\nabla u \in C^0(\overline\Om^g)$. Note that
	$$\divg_g(ue^{-f}F\pphi_\eps)=e^{-f}g(\nabla u,F)\pphi_\eps+u\divg_g(e^{-f}F\pphi_\eps) \, .$$
	Then, we have 
	\begin{align*}
		\int_\Om u \, \divg_g(e^{-f}F\pphi_\eps) \, dV_g &=
		\int_{\Om\setminus B_\eps^g} u \, \divg_g(e^{-f}F\pphi_\eps) \, dV_g \\
		&=\int_{\Om\setminus B_\eps^g} \divg_g(ue^{-f}F\pphi_\eps) \, dV_g-\int_{\Om\setminus B_\eps^g} e^{-f}g(\nabla u,F)\pphi_\eps \, dV_g \, .
	\end{align*}
	Since the divergence theorem applies in $\Om\setminus B_\eps^g$, it follows that
	\begin{equation}\label{reldivu}
		\int_\Om u \, \divg_g(e^{-f}F\pphi_\eps) \, dV_g =\int_{\p\Om} ue^{-f} \,g(F,\nu)\, d\tilde V_g \, - \int_{\Om\setminus B_\eps^g} e^{-f}\,g(\nabla u,F)\pphi_\eps \, dV_g \, .
	\end{equation}
	Thus, it suffices to prove that
	\begin{equation}\label{limintdivgu}
		(i) \quad\int_\Om u \,\divg_g(e^{-f}F\pphi_\eps) \, dV_g \to \int_\Om u \, \divg_g(e^{-f}F) \, dV_g \quad \text{as} \ \eps\to 0 
	\end{equation}
	and
	\begin{equation}\label{limintdivgnabla}
		(ii) \quad\int_{\Om\setminus B_\eps^g} e^{-f} g(\nabla u,F)\pphi_\eps \, dV_g \to \int_\Om e^{-f} g(\nabla u,F) \, dV_g  \quad \text{as} \ \eps\to 0 \, .
	\end{equation}
	\textbf{Proof of \textit{(i)}} : Note that 
	$$\divg_g(e^{-f}F\pphi_\eps)= e^{-f}\left\{g(F,\nabla \pphi_\eps)+\left(\divg_gF-g(\nabla f,F)\right)\pphi_\eps\right\} \quad \text{in} \ \Om \, .$$
	Hence, recalling that $d\mu_g = e^{-f}dV_g$, we get
	\begin{equation}\label{intdivgu}
		\int_{\Om\setminus B_\eps^g} u \, \divg_g(e^{-f}F \pphi_\eps) \, dV_g = \int_{\Om\setminus B_\eps^g} u \, g(F,\nabla \pphi_\eps) \, d\mu_g +\int_{\Om\setminus B_\eps^g} u \, (\divg_g F-g(\nabla f,F)) \, \pphi_\eps \, d\mu_g \, .
	\end{equation}
	Thus, to establish \eqref{limintdivgu} it is enough to show that
	\begin{equation}\label{limto0u}
		\int_{\Om\setminus B_\eps^g} u \, g(F,\nabla \pphi_\eps) \, d\mu_g \to 0  \quad \text{as} \ \eps \to 0 
	\end{equation}
	and
	\begin{equation}\label{limtodivgu}
		\begin{multlined}
			\int_{\Om\setminus B_\eps^g} u \,(\divg_g F-g(\nabla f,F))\, \pphi_\eps \, d\mu_g \to \int_{\Om}(\divg_g F-g(\nabla f,F)) \, d\mu_g =\int_\Om \divg_g(e^{-f}F)\, dV_g \\ \text{as } \eps \to 0 \, .
		\end{multlined}
	\end{equation}
	We first verify \eqref{limto0u}. Since $\nabla \pphi_\eps$ is supported in $B_{2\eps}^g\setminus B_\eps^g$, $u \in L^\infty(\Om)$, $|F|_g \in L^2(\Om,d\mu_g)$ and $|\nabla \pphi_\eps|_g\leq \frac 2 \eps$, we have
	\begin{align*}
		\int_{\Om\setminus B_\eps^g} u \, g(F,\nabla \pphi_\eps) \, d\mu_g&=\int_{\Om} u \, g(F,\nabla \pphi_\eps)\chi_{\Om\setminus B_\eps^g} \, d\mu_g \\
		&\leq \|u\|_{L^\infty(\Om)}\left(\int_{\Om} |F|_g^2 \, d\mu_g\right)^{1/2} \left(\int_\Om|\nabla \pphi_\eps|_g^2 \, d\mu_g \right)^{1/2}\\
		&\leq 4\|u\|_{L^\infty(\Om)}\left(\int_{\Om} |F|_g^2 \, d\mu_g\right)^{1/2} \left(\frac{1}{\eps^2}|B_{2\eps}^g\setminus B_{\eps}^g|_f\right)^{1/2}
	\end{align*}
	From \eqref{wvolumeball}, we have 
	$$\frac{1}{\eps^2}|B_{2\eps}^g\setminus B_{\eps}^g|_f\sim \eps^{\alpha+n-2} \to 0 \quad \text{as} \ \eps \to 0 \, ,$$
	since by assumption $$\alpha>2-n\, .$$
	\indent To prove \eqref{limtodivgu}, note that:
	\begin{equation}\label{secdivu}
		\int_{\Om\setminus B_\eps^g} u \, (\divg_g F-g(\nabla f,F))\, \pphi_\eps \, d\mu_g=\int_{\Om} u \, (\divg_g F-g(\nabla f,F)) \, \pphi_\eps \, \chi_{\Om\setminus B_\eps^g} \, d\mu_g \, .
	\end{equation}
	Recalling that $\pphi_\eps\leq 1$, we have
	\begin{align*}
		|u \, (\divg_g F-g(\nabla f,F)) \, \pphi_\eps \, \chi_{\Om\setminus B_\eps^g}|&\leq \|u\|_{L^\infty(\Om)} \, |\divg_g F-g(\nabla f,F)|\\
		&\leq \|u\|_{L^\infty(\Om)}(|\divg_g F|+|\nabla f|_g |F|_g) \in L^1(\Om,d\mu_g) \, ,
	\end{align*}
	since $|\divg_g F| \in L^1(\Om,d\mu_g), |F|_g \in L^2(\Om,d\mu_g)$ and $|\nabla f|_g \in L^2(\Om,d\mu_g)$ under condition \eqref{condgammadivpoho}, as verified in \eqref{nablafinl2}. Then, we can apply the dominated convergence theorem to \eqref{secdivu}, obtaining \eqref{limintdivgu}.\smallskip\\
	\textbf{Proof of \textit{(ii)}} : Note that
	\begin{align}\label{terzdivu}
		\int_{\Om\setminus B_\eps^g} e^{-f}g(\nabla u,F)\pphi_\eps \, dV_g = \int_{\Om} g(\nabla u,F)\pphi_\eps \chi_{\Om\setminus B_\eps^g}\, d\mu_g
	\end{align}
	Recalling that $u \in L^\infty(\Om)$ and $\pphi_\eps\leq 1$, we observe
	$$|g(\nabla u,F)\pphi_\eps \chi_{\Om\setminus B_\eps^g}|\leq |\nabla u|_g^2 |F|_g^2 \in L^1(\Om,d\mu_g) \, ,$$
	since $|F|\in L^2(\Om,d\mu_g)$ and $u \in W_0^{1,2}(\Om,d\mu_g)\subset W^{1,2}(\Om,d\mu_g)$.
	Thus, by applying the dominated convergence theorem to \eqref{terzdivu}, we conclude that \eqref{limintdivgnabla} holds as well.
\end{proof}
We now establish a Pohozaev-type identity adapted to our setting, derived from \eqref{pohozaev}. 
\begin{proposition}
	Let $\Om$ be a domain satisfying \eqref{domain}, and let $u \in W_0^{1,2}(\Om,d\mu_g)\cap L^\infty(\Om)$ be a solution to \eqref{ourprob} such that $|\nabla u|_g^2 \in L^\infty(\Om)$. Assume that $\alpha$ and $\gamma$ satisfy \eqref{condgammadivpoho}. Then,
	\begin{equation}\label{pohozaev_nc}
		\int_\Om u \,d\mu_g = \frac{n+\alpha}{n+\alpha+2}c^2 |\Om|_f \, .
	\end{equation}
\end{proposition}
\begin{remark}\label{choiceV_nc}
	Define
	\begin{equation}\label{V_nc}
		V(x) \coloneqq \frac{d_g(x,O)^2}{2(n+\alpha)} \, .
	\end{equation}
	Recalling that
	$$d_g(x,O)=\frac{|x|^{1-\gamma}}{1-\gamma}$$ from \eqref{grad}, \eqref{hess} and \eqref{lap} we obtain
	\begin{equation}\label{nablaV_nc}
		\nabla V =\frac{x}{(1-\gamma)(n+\alpha)} \, ,
	\end{equation}
	\begin{equation}\label{hessV_nc}
		\nabla^2 V=\frac{1}{(n+\alpha)}\,g\, ,
	\end{equation}
	and
	\begin{equation}\label{lapwV_nc}
		\lapw V=1 \, .
	\end{equation}
	Hence, $V$ satisfies the assumptions of \Cref{pohozaevprop}, and the argument used in the proof of \Cref{pohozaevprop} can be repeated in the same way in this case. The only point requiring verification is that \Cref{lemmadivthm} may be applied in place of the divergence theorem, which holds under suitable regularity assumptions on $u$, as we will show below. 
\end{remark}
\begin{proof}
	Recall that $u\in C^\infty(\Om)\cap C^1(\overline\Om^g)$. As mentioned in \Cref{choiceV_nc}, we can follow the proof of \Cref{pohozaevprop}, with $V$ defined as in \eqref{V_nc}; the only point to check is that \Cref{lemmadivthm} justifies the application of the divergence theorem. In particular, we apply it after integrating the differential equality \eqref{divthmpoho1}, and directly in the integral identity \eqref{poho2}. \\Using the notation of \Cref{lemmadivthm}, the vector field in the first case is
	$$F=\frac{|\nabla u|_g^2}{2}\, \nabla V-g(\nabla V,\nabla u)\nabla u \, .$$
	Note that, from \eqref{nablaV_nc} and $u \in C^1(\overline\Om^g)$, it follows $F\in C^0(\overline\Om^g)$. We verify that $|F|_g \in L^2(\Om,d\mu_g)$ and $\divg_gF\in L^1(\Om,d\mu_g)$. Note that
	$$|F|_g\leq \frac 3 2 \||\nabla u|_g^2\|_{L^\infty(\Om)} |\nabla V|_g \, ,$$
	and thus, from \eqref{nablaV_nc} we obtain
	\begin{equation}\label{Fineq}
		|F|_g\leq \frac {3}{2(n+\alpha)} \||\nabla u|_g^2\|_{L^\infty(\Om)} |x|_g \, .
	\end{equation}
	Since $g_{ij}=|x|^{-2\gamma}\delta_{ij}$, 
	we have
	\begin{equation}\label{modulusx}
		|x|_g =|x|^{1-\gamma} \, ,
	\end{equation}
	and hence from \eqref{dist0} it follows that
	\begin{equation}\label{distmodulus}
		d_g(x,O)=\frac{|x|_g}{1-\gamma} \, .
	\end{equation}
	Thus, since $\Om$ is bounded with respect to $g$ and $|\Om|_f$ is finite, we conclude that 
	\begin{equation}\label{modulusxl2}
		|x|_g\in L^2(\Om,d\mu_g)\, ,
	\end{equation}
	which, together with \eqref{Fineq}, implies
	$$|F|_g \in L^2(\Om,d\mu_g) \, .$$ 
	Next, note that 
	$$\divg_g F =e^{f}\divg_g (e^{-f}F)+g(F,\nabla f) \, ,$$
	so that, using \eqref{divthmpoho1} with $\lapw V=1$, we obtain
	\begin{equation}\label{divF}
		\divg_g F =\frac{|\nabla u|_g^2}{2}\left(1+g(\nabla V,\nabla f)\right)-\nabla^2V(\nabla u,\nabla u)+g(\nabla V,\nabla u)(1-g(\nabla u,\nabla f))\, .
	\end{equation}
	Since $f(x)=-\alpha(1-\gamma) \log |x|-\alpha\log(1-\gamma)$, from \eqref{grad} and \eqref{modulusx} it follows that
	\begin{equation}\label{nablaf}
		\nabla f = -\alpha(1-\gamma)\frac{x}{|x|_g^2} \, .
	\end{equation}
	Substituting \eqref{nablaV_nc}, \eqref{hessV_nc} and \eqref{nablaf} into \eqref{divF}, we get
	$$|\divg_g F| \leq \frac{|\nabla u|_g^2}{2}\left(1+\frac{\alpha}{n+\alpha}\right)+\frac{1}{(n+\alpha)(1-\gamma)}\,|x|_g |\nabla u|_g +\frac{\alpha+1}{(n+\alpha)}\, |\nabla u|_g^2 \in L^1(\Om,d\mu_g) \, ,$$
	since both $|\nabla u|_g, |x|_g$ belong to $L^2(\Om,d\mu_g)$. Thus, \Cref{lemmadivthm} applies.\\
	It remains to justify the use of the divergence theorem in identity \eqref{poho2}. In this case, 
	$$F=\nabla V$$
	and from \eqref{nablaV_nc} together with \eqref{modulusxl2} it immediately follows that $|F|_g\in L^2(\Om,d\mu_g)$. Moreover,
	$$\divg_g F=\divg_g(\nabla V)=\frac{\divg_g(x)}{n+\alpha} \, .$$
	Using \eqref{div} and recalling that $\phi(x)=-\gamma \log|x|$, we obtain 
	$$\divg_g F=\frac{n(1-\gamma)}{n+\alpha}\, .$$ 
	Therefore, we have $|\divg_g F| \in L^1(\Om,d\mu_g)$ and \Cref{lemmadivthm} applies in this case as well. \\
	Hence, we can repeat the proof of \Cref{pohozaevprop} to obtain the Pohozaev identity \eqref{pohozaev}. Substituting \eqref{hessV_nc} into \eqref{pohozaev} yields
	$$c^2|\Om|_f= \int_\Om u\, d\mu_g +\frac{2}{n+\alpha}\int_\Om u \, d\mu_g \, ,$$
	which concludes the proof.
\end{proof}
Finally, we discuss the nature of the energy space $W^{1,2}_0(\Om,d\mu_g)$. Recall that it is defined as the closure of $C^\infty_c(\Om)$ with respect to the norm in \eqref{weightnorm}. When the origin belongs to $\overline \Om$, one may wonder whether adopting this functional framework imposes any additional constraint on the behavior of solutions to \eqref{ourprob} near the origin. Nevertheless, this is not the case under certain conditions on the parameters $\alpha$ and $\gamma$, as stated in the following proposition.
\begin{proposition}\label{densità}
	Let $\Om$ satisfy \eqref{domain} and assume that $O \in \overline \Om$.  If $\alpha$ and $\gamma$ satisfy \eqref{condgammadivpoho}, then $$W_0^{1,2}(\Om,d\mu_g)=W_0^{1,2}(\Om\cup\{O\},d\mu_g) \, .$$
\end{proposition}
\begin{proof}
	The inclusion $W^{1,2}_0(\Om,d\mu_g)\subseteq W^{1,2}_0(\Om\cup \{O\},d\mu_g)$ is obvious. To obtain the other one, it is enough to prove that $\forall u \in C^\infty_c(\Om \,\cup\, \{O\})$, $\forall \eps>0$ there exists $g_\eps=g(u,\eps)\in C^\infty_c(\Om)$ such that
	$$\|u-g_\eps\|_{W_0^{1,2}(\Om, d\mu_g)}<\eps \, .$$
	Thus, let $u\in C^\infty_c(\Om\cup\{O\})$. Fix $\eps>0$. Suppose that $B_{4\eps}^g$ is contained in $\Om$, possibly after choosing a smaller $\eps$, where $B_{4\eps}^g$ is defined as in \eqref{defball}. Let $\varphi_\eps \in C^\infty_c(\Om\cup\{O\})$ be a cut-off function with $\varphi_\eps \equiv 0$ in $\Om\setminus B_{2\eps}^g$, $\varphi_\eps \equiv 1$ in $B_\eps^g$, $0\leq \varphi_\eps \leq 1$, and $|\nabla \varphi_\eps|_g\leq \frac 2 \eps$. We choose $g_\eps=(1-\varphi_\eps)u \in C^\infty_c(\Om)$ and we observe that
	\begin{align}
		\|u-g_\eps\|_{W_0^{1,2}(\Om, d\mu_g)}^2&=\|u\varphi_\eps\|_{W_0^{1,2}(\Om, d\mu_g)}^2\nonumber\\
		&=\int_\Om u^2 \varphi_\eps^2 \, d\mu_g +\int_\Om |\nabla(u\varphi_\eps)|^2 \, d\mu_g \, . \label{density}
	\end{align}
	If both the integral terms in \eqref{density} vanish as $\eps \rightarrow 0$, then the thesis follows. \\
	Hence, recalling that $u\in L^\infty(\Om)$, $\varphi_\eps \leq 1$ and $\varphi_\eps \equiv 0$ in $\Om\setminus B_{2\eps}^g$, we have 
	$$\int_\Om \varphi_\eps^2 u^2 \, d\mu_g \leq \|u\|_{L^\infty(\Om)}^2\int_{B_{2\eps}^g}\, d\mu_g = \|u\|_{L^\infty(\Om)}^2\int_{B_{2\eps}^g} e^{-f}\, dV_g  \rightarrow 0 \quad \text{as} \ \eps \rightarrow 0 \, ,$$
	since $e^{-f} \in L^1(\Om,dV_g)$.\\
	Moreover, since $|\nabla \varphi_\eps|_g\leq \frac 2 \eps$ and $|\nabla u|_g^2\in L^\infty(\Om)$, we obtain that
	\begin{equation}
		\begin{aligned}\label{nablaupphi}
			\int_\Om |\nabla(u\varphi_\eps)|_g^2 \, d\mu_g &= \int_{B_{2\eps}^g\setminus B_{\eps}^g} |\nabla\varphi_\eps|_g^2 \, u^2 \, d\mu_g+\int_{B_{2\eps}^g}\varphi_\eps^2 |\nabla u|_g^2 \,  d\mu_g\\
			&\leq 4\|u\|_{L^\infty(\Om)}^2 \frac{1}{\eps^2}|B_{2\eps}^g\setminus B_{\eps}^g|_f +\|\nabla u\|_{L^\infty(\Om)}|B_{2\eps}^g|_f \, .
		\end{aligned}
	\end{equation}
	From \eqref{wvolumeball}, we have 
	\begin{equation*}
		\frac{1}{\eps^2}|B_{2\eps}^g\setminus B_{\eps}^g|_f\sim \eps^{\alpha+n-2} \to 0 \quad \text{as} \ \eps \to 0 \, ,
	\end{equation*}
	since, by assumption, $\alpha>2-n$.
	Moreover, from $e^{-f} \in L^1(\Om,dV_g)$ it follows that 
	$$|B_{2\eps}^g|_f \to 0 \quad \text{as} \ \eps \to 0 \, .$$
	Thus, from \eqref{nablaupphi} we conclude that
	$$\int_\Om |\nabla (u\pphi_\eps)|_g^2 \, d\mu_g \to 0 \quad \text{as} \ \eps \to 0 \, ,$$
	and the proof is complete.
\end{proof}
\subsection{Proof of Theorem \ref{main_nc}}\label{proofsect}
The proof is similar to that of \Cref{main} in the compact case. In \Cref{choiceV_nc} we noted that
$$V(x)=\frac{d_g(x,O)^2}{n+\alpha}$$ satisfies $\lapw V=1$ and 
\begin{equation}\label{hessVhp_nc}
	\nabla^2 V=\frac \beta n g \, , \quad \text{with} \ \beta=\frac{n}{n+\alpha} \, .
\end{equation}
As in the compact case, we start from the differential identity \eqref{idciraolo}, integrate and apply the Pohozaev identity \eqref{pohozaev_nc} - which corresponds to \eqref{pohozaev} together with \eqref{hessVhp_nc}. Under the parameter's range \eqref{parcond}, we have $\beta<1$, and thus the weighted Young inequality \eqref{youngw} can be applied to recover \eqref{ultima}. Moreover, assumption \eqref{ricwgeq0} is satisfied under \eqref{parcond}, so that, as in the compact case, we deduce \eqref{scalarprod} and \eqref{hessu}, from which we can conclude.\\
Although the proof begins in the same way, we provide all details for completeness. Indeed, since $\overline\Om^g$ may not be compact, as observed in \Cref{preliminarynoncompl}, in this case we must justify the use of the divergence theorem by appealing to \Cref{lemmadivthm}. Moreover, we will constructively deduce the range \eqref{parcond}. 
\begin{proof}[Proof of \Cref{main_nc}]
	Assume $\gamma<1$. This condition is imposed only to guarantee that the origin $O$ lies at finite distance, as we observed in \Cref{preliminarynoncompl}. Furthermore, we assume that 
	\begin{equation}\label{gammainiz}
		\alpha>2-n \, ,
	\end{equation}
	in order to apply the divergence theorem given in \Cref{lemmadivthm} and the Pohozaev identity \eqref{pohozaev_nc}.\\
	Recall that $u \in C^2(\Om)\cap C^1(\overline \Om^g)$. We observe that in $\Om$ it holds
	\begin{align}
		& |\nabla ^2u|_g^2-(\lapw u)^2 \nonumber\\
		&= |\nabla ^2 u|_g^2-(\lap_g u-g(\nabla f,\nabla u))^2\nonumber\\
		&=\left|\nabla ^2 u-\frac{\lap_g u}{n}g\right|_g^2+2\,g(\nabla f,\nabla u)\lap_g u-g(\nabla f,\nabla u)^2-\frac{n-1}{n}(\lap_g u)^2\nonumber\\
		&=\left|\nabla ^2 u-\frac{\lap_g u}{n}g\right|_g^2-\frac{n-1}{n} \Big(\lapw u+g(\nabla f,\nabla u)\Big)^2+2\,g(\nabla f,\nabla u)\lap_g u-g(\nabla f,\nabla u)^2 \, . \nonumber 
	\end{align}
	Using the identity
	$$\lap_g u= \lapw u+g(\nabla f,\nabla u) \, ,$$
	we deduce that
	\begin{equation}\label{id_nc}
		\begin{aligned}
			& |\nabla ^2u|_g^2-(\lapw u)^2\\
			&=\left|\nabla ^2 u-\frac{\lap_g u}{n}g\right|_g^2-\frac{n-1}{n}(\lapw u)^2 +\frac 2 n \,g(\nabla f,\nabla u)\lapw u+\frac 1 n \, g(\nabla f,\nabla u)^2 \, .
		\end{aligned}
	\end{equation}
	Thus, integrating identity \eqref{idciraolo} and using \eqref{id_nc}, we obtain
	\begin{equation*}
		\begin{split}
			&\int_\Om \divg_g \Big\{e^{-f}\left(u\nabla \left(\frac{|\nabla u|_g^2}{2}\right)-u\lapw u \nabla u-\frac{|\nabla u|_g^2}{2}\nabla u\right)\Big\}\,dV_g\\
			&=\int_\Om e^{-f}u\,\Big\{\left|\nabla ^2 u-\frac{\lap_g u}{n}g\right|_g^2-\frac{n-1}{n}(\lapw u)^2+\frac 2 n \, g(\nabla f,\nabla u)\lapw u +\frac 1 n \, g(\nabla f,\nabla u)^2 \Big\}\, dV_g\\
			&+\int_\Om e^{-f}\Big\{-\frac 3 2 |\nabla u|_g^2 \, \lapw u + u \, (\ric_g+\nabla ^2 f)[\nabla u,\nabla u]\Big\}\, dV_g \, .
		\end{split}
	\end{equation*}
	We want to apply \Cref{lemmadivthm} to the left-hand side of the above identity. For this,
	let us define the vector field 
	\begin{equation}\label{F}
		F=u\nabla \left(\frac{|\nabla u|_g^2}{2}\right)-u\lapw u \nabla u-\frac{|\nabla u|_g^2}{2}\nabla u \, .
	\end{equation}
	Since $u \in C^1(\overline\Om^g)$, it follows that $F \in C^0(\overline\Om^g)$. Recalling that $\lapw u\equiv -1$ in $\Om$, and using the assumptions $u \in L^\infty(\Om)$, $|\nabla u|_g^2 \in L^\infty(\Om)$, we estimate
	$$|F|_g \leq \|u\|_{L^\infty(\Om)}\left(\frac 1 2 |\nabla |\nabla u|_g^2 |_g+|\nabla u|_g\right)+\frac 1 2 \||\nabla u|_g^2\|_{L^\infty(\Om)}|\nabla u|_g \, .$$
	Since $\nabla |\nabla u|_g^2 \in L^2(\Om,d\mu_g)$ and $u\in W_0^{1,2}(\Om,d\mu_g)\subset W^{1,2}(\Om,d\mu_g)$, we conclude that $|F|_g \in L^2(\Om,d\mu_g)$.
	Suppose now that the right-hand side of \eqref{idciraolo} is nonnegative, so that $\divg_g (e^{-f}F)\geq 0$; we will verify this later by explicitly determining the conditions on the parameters $\alpha$ and $\gamma$ ensuring it holds. \\Thus, we may apply \Cref{lemmadivthm} and we find
	\begin{align}
		-\int_{\p\Om} \frac{|\nabla u|_g^2}{2}\, g(\nabla u,\nu)\,d\tilde \mu_g &= \int_\Om \Big( -\frac{n-1}{n}\,u+\frac 3 2 |\nabla u|_g^2\Big)\, d\mu_g \nonumber\\
		&\begin{multlined}[t]+\int_\Om u \, \Big\{ \Big|\nabla ^2 u-\frac{\lap_g u}{n}g\Big|_g^2 +\frac 2 n \,g(\nabla f,\nabla u)\lapw u \\+\frac 1 n \,g(\nabla f,\nabla u)^2 + (\ric_g+\nabla ^2 f)[\nabla u,\nabla u]\Big\}\, d\mu_g \, . \label{prima_nc}
		\end{multlined} 
	\end{align}
	Using \eqref{cweight_nc} and the fact that $u_\nu =-c$ on $\p\Om$, the left-hand side of \eqref{prima_nc} simplifies to $\frac {c^2}{2}|\Om|_f$. Moreover, from \eqref{intnablau_nc}, we have that
	\begin{align*}
		\int_\Om \Big( -\frac{n-1}{n}u+\frac 3 2 |\nabla u|_g^2 \Big) \, d\mu_g = \frac{n+2}{2n}\int_\Om u \, d\mu_g \, .
	\end{align*}
	Substituting into \eqref{prima_nc}, we get
	\begin{equation}\label{seconda_nc}
		\begin{multlined}
			\frac{c^2}{2}|\Omega|_f = \frac{n+2}{2n} \int_\Omega u \, d\mu_g 
			+ \int_\Omega u \, \Big\{ 
			\left|\nabla^2 u - \frac{\Delta_g u}{n} g \right|_g^2 
			+ \frac{2}{n} g(\nabla f, \nabla u)\lapw u \\
			+ \frac{1}{n} g(\nabla f, \nabla u)^2 
			+ (\mathrm{Ric}_g + \nabla^2 f)[\nabla u, \nabla u] 
			\Big\} \, d\mu_g \, . 
		\end{multlined}
	\end{equation}
	Next, applying the Pohozaev identity \eqref{pohozaev_nc}, we get
	\begin{equation}\label{postpoho_nc}
		\begin{multlined}
			0=\frac{\alpha}{n(n+\alpha)} \int_\Om u \, d\mu_g + \int_\Omega u \, \Big\{ 
			\left|\nabla^2 u - \frac{\Delta_g u}{n} g \right|_g^2 
			+ \frac{2}{n} g(\nabla f, \nabla u)\lapw u \\
			+ \frac{1}{n} g(\nabla f, \nabla u)^2 
			+ (\mathrm{Ric}_g + \nabla^2 f)[\nabla u, \nabla u] 
			\Big\} \, d\mu_g \, . 
		\end{multlined}
	\end{equation}
	As already mentioned, the argument up to this point follows closely that of the compact case. Indeed, identity \eqref{postpoho_nc} corresponds to relation \eqref{princ} in the compact setting, with
	\begin{equation}\label{beta_nc}
		\beta =\frac{n}{n+\alpha} \, .
	\end{equation}
	However, while \eqref{princ} holds only as an inequality, here the equality is achieved in \eqref{postpoho_nc}. This follows from the fact that, in the present setting, \eqref{hessV} holds as an identity, as shown in \eqref{hessVhp_nc}.\\
	Proceeding as in the compact case, we apply the weighted Young inequality $2|ab|\leq k a^2+\frac 1 k b^2$ with $a=\frac{1}{\sqrt n}\lapw u$, $b=\frac{1}{\sqrt n} g(\nabla f,\nabla u)$, and, from \eqref{postpoho_nc}, we find
	\begin{equation}\label{youngw_nc}
		\frac 2 n \, g(\nabla f,\nabla u)\lapw u \geq -\frac k n (\lapw u)^2-\frac{1}{kn}g(\nabla f,\nabla u)^2 \, ,
	\end{equation}
	where $k>0$ will be chosen appropriately. Note that the equality is attained if and only if 
	$$k \cdot \lapw u = -g(\nabla f,\nabla u) \, ,$$ 
	i.e., 
	$$k=g(\nabla f,\nabla u) \, .$$
	Thus, plugging \eqref{youngw_nc} into \eqref{postpoho_nc}, we obtain
	\begin{equation}\label{postyoung_nc}
		\begin{multlined}[t]
			0 \geq  \left(\frac{\alpha}{n(n+\alpha)}-\frac k n \right)\int_\Om u \, d\mu_g +\int_\Om u\, \Big|\nabla ^2 u-\frac{\lap_g u}{n}g\Big|_g^2 
			+\int_\Om u\, \Big\{\frac 1 n\Big(1-\frac{1}{k}\Big)\,g(\nabla f,\nabla u)^2\\+(\ric_g +\nabla^2f)[\nabla u,\nabla u]\Big\}\,d\mu_g \, . 
		\end{multlined}
	\end{equation}
	Now, we choose $k > 0$ such that the constant multiplying $\displaystyle\int_\Om u\, d\mu_g$ in \eqref{postyoung_nc} vanishes, that is
	\begin{equation}\label{k_nc}
		k=\frac{\alpha}{n+\alpha}\, .
	\end{equation}
	Recall that $k$ has to be positive, as it represents the coefficient used in the Young inequality \eqref{youngw_nc}: it turns out that 
	\begin{equation}\label{kpositive}
		\alpha>0 \quad \mbox{or} \quad \alpha<-n \, .
	\end{equation}
	Thus, combining \eqref{kpositive} and \eqref{gammainiz}, we get 
	\begin{equation}\label{alphagamma1}
		\alpha>0 \quad \text{and} \quad \gamma<1 \, .
	\end{equation}
	Note that this implies $0 < \beta < 1$, where $\beta$ is defined by \eqref{beta_nc}. From \eqref{postyoung_nc}, and observing that $$\frac 1 n \left(1-\frac 1 k\right)=-\frac{1}{\alpha} \, , \quad \text{and} \quad g(\nabla f,\nabla u)^2=(df \otimes df) [\nabla u,\nabla u] \, ,$$
	we get
	\begin{equation}\label{ultima_nc}
		0 \geq \int_\Om u \,\Big|\nabla ^2 u-\frac{\lap_g u}{n}g\Big|_g^2 \, d\mu_g  \\
		+\int_\Om u \, (\ric_g +\nabla^2f-\frac{1}{\alpha} \, df \otimes df)[\nabla u,\nabla u]\,d\mu_g \, . 
	\end{equation}
	Note that, since $u$ is nonnegative in $\Om$ and 	
	\begin{equation}\label{hessu_nc}
		\Big |\nabla^2 u -\frac{\lap_g u}{n}g\Big |_g \geq 0\,,
	\end{equation}
	the first integral on the right hand side in \eqref{ultima_nc} is nonnegative. Our strategy aims to find conditions on parameters $\alpha$ and $\gamma$ such that 
	\begin{equation}\label{matrixgeq0}
		\ric_g +\nabla ^2f-\frac{1}{\alpha} \, df \otimes df 
		\geq 0 \, ,
	\end{equation} 
	which then ensures that
	\begin{equation}\label{matrixgeq0uu}
	(\ric_g +\nabla ^2f-\frac{1}{\alpha} \, df \otimes df )[\nabla u,\nabla u]\geq 0 \, .
	\end{equation} 
	This will enable us to estimate both the integral terms on the right hand side of  \eqref{ultima_nc} from below by zero. Recalling that $u>0$ in $\Om$, this implies that equality is attained in all the inequalities \eqref{youngw_nc}, \eqref{hessu_nc}, and \eqref{matrixgeq0uu}, and this will lead to the conclusion.\\ Note that this also ensures that $\divg_g F\geq 0$. Indeed, from \eqref{idciraolo} we have $$\divg_g (e^{-f}F) =e^{-f}\left\{u\left[|\nabla ^2u|_g^2-(\lapw u)^2\right]-\frac 3 2 |\nabla u|^2\lapw u+u\,(\ric_g +\nabla ^2f)[\nabla u,\nabla u]\right\} \, ,$$
	which coincides with the integrand in \eqref{seconda_nc}. Then, under condition \eqref{matrixgeq0}, it is indeed nonnegative: we had previously estimated it from below using a weighted Young’s inequality, and we now impose that the resulting lower bound is nonnegative. 
	\par From now on, all the equalities will be set in $\Om$. Let's verify when \eqref{matrixgeq0} holds. Recalling \eqref{ric}, \eqref{hess} and $\phi(x)=-\gamma\ln |x|$, $f(x)=-\alpha \ln |x|$, straightforward calculations show that 
	\begin{align*}
		&(\ric_g)_{ij} = -\frac{(n-2)\gamma (\gamma -2)}{|x|^2}\Big(\delta_{ij}-\frac{x_ix_j}{|x|^2}\Big) \, ,\\
		&(\nabla ^2 f)_{ij}= -\frac{\alpha(1-\gamma)^2}{|x|^2}\Big(\delta_{ij}-2\,\frac{x_ix_j}{|x|^2}\Big) \, ,\\
		&(df\otimes df)_{ij}=\frac{\alpha^2(1-\gamma)^2}{|x|^4}x_ix_j \, .
	\end{align*}
	Therefore
	\begin{equation}\label{matrix}
		\Big(\ric_g +\nabla ^2f-\frac{1}{\alpha}\, df \otimes df \Big)_{ij}
		=-\frac{(n-2)\gamma(\gamma-2)+\alpha(1-\gamma)^2}{|x|^2}\Big(\delta_{ij}-\frac{x_ix_j}{|x|^2}\Big) \, ,
	\end{equation}
	and \eqref{matrixgeq0} becomes 
	\begin{equation}\label{alphagamma}
		(n-2)\gamma (\gamma-2)+\alpha (1-\gamma)^2\leq 0\, .
	\end{equation}
	Hence, \eqref{alphagamma1} together with \eqref{alphagamma} lead to 
	\begin{equation*}
		\alpha>0 \quad \text{and} \quad \gamma_1 \leq \gamma < 1 \, ,
	\end{equation*}
	where $$\gamma_1 = 1-\sqrt{\frac{n-2}{\alpha+n-2}} \, ,$$ 
	i.e., the range given by \eqref{parcond}; note that this is nonempty, since $\gamma_1<1$. Thus, from now on we assume \eqref{parcond} satisfied.\\
	From \eqref{youngw_nc}, \eqref{matrixgeq0} and \eqref{hessu_nc}, we find respectively
	\begin{equation}\label{firstcond_nc}
		g(\nabla f,\nabla u)=\frac{\alpha}{n+\alpha}\,,
	\end{equation}
	\begin{equation}\label{secondcond_nc}
		-\frac{(n-2)\gamma(\gamma-2)+\alpha(1-\gamma)^2}{|x|^2}\Big(\delta_{ij}-\frac{x_ix_j}{|x|^2}\Big)(\nabla u)^i (\nabla u)^j= 0\, ,
	\end{equation}
	and
	\begin{equation*}
		\nabla ^2 u=\frac{\lap_g u}{n}g \, .
	\end{equation*}
	From \eqref{firstcond_nc}, we get 
	\begin{equation}\label{nablau_nc}
		\nabla u = -\frac{x}{(1-\gamma)(n+\alpha)}\, ,
	\end{equation}
	and from \eqref{grad} it follows that 
	$$Du= - \frac{|x|^{-2\gamma}x}{(1-\gamma)(n+\alpha)} \, .$$
	Hence, we deduce that $u$ has the form 
	$$	u(x)=C-\frac{|x|^{2-2\gamma}}{2(1-\gamma)^2(n+\alpha)} \quad \text{in} \ \Om \, ,$$
	for some constant $C$. Using \eqref{dist0}
	and imposing the boundary condition $u=0$ on $\partial\Om$, we conclude that 
	\begin{equation}\label{solution}
		u(x)=\frac{R^2-d_g(x,O)^2}{2(n+\alpha)} \quad \text{in} \ \Om \, ,
	\end{equation}
	and $\Om=B_R^g$.
	\par Now, we have to verify that $u$ given by \eqref{solution} belongs to the energy space $W^{1,2}_0(\Om,d\mu_g)$ and satisfies the regularity assumptions stated in \eqref{regassu}. Since $\Om$ is a ball of radius $R$ centered at the origin, it follows that $u \in L^\infty(\Om)$. Moreover, using \eqref{distmodulus} and the expression for $\nabla u$ given in \eqref{nablau_nc}, we find that
	$$|\nabla u|_g^2 =\frac{d_g(x,O)^2}{(n+\alpha)^2} \in L^\infty(\Om) ,$$
	and thus, using \eqref{grad} and \eqref{dist0}, we have
	$$\nabla |\nabla u|_g^2 =\frac{2x}{(1-\gamma)(n+\alpha)^2} \in L^2(\Om,d\mu_g) \, ,$$
	since $|x|_g \in L^2(\Om,d\mu_g)$, as already noted in \eqref{modulusxl2}. \\Finally, we show that $u\in W^{1,2}_0(\Om,d\mu_g)$. The proof follows the same argument as in \Cref{densità}. Let $\eps>0$ be fixed small enough such that $B_{4\eps}^g$ is contained in $\Om$, where $B_{4\eps}^g$ is defined as in \eqref{defball}. Let $\varphi_\eps \in C^\infty_c(\Om)$ be a cut-off function with $\varphi_\eps \equiv 0$ in $\Om\setminus B_{2\eps}^g$, $\varphi_\eps \equiv 1$ in $B_\eps^g$, $0\leq \varphi_\eps \leq 1$, and $|\nabla \varphi_\eps|_g\leq \frac 2 \eps$. We choose $g_\eps=(1-\varphi_\eps)u \in C^\infty_c(\Om)$ and we observe that
	\begin{align}
		\|u-g_\eps\|_{W_0^{1,2}(\Om,d\mu_g)}^2&=\|u\varphi_\eps\|_{W_0^{1,2}(\Om,d\mu_g)}^2\nonumber\\
		&=\int_\Om u^2 \varphi_\eps^2 \, d\mu_g +\int_\Om |\nabla(u\varphi_\eps)|_g^2 \, d\mu_g \, . \label{density1}
	\end{align}
	If both the integral terms in \eqref{density1} vanish as $\eps \to 0$, then the thesis follows. \\ Recalling that $u \in L^\infty(\Om)$ and $u>0$ in $\Om$, we deduce that $u^2\in L^\infty(\Om)$.
	Hence, since $\varphi_\eps \leq 1$ and $\varphi_\eps \equiv 0$ in $\Om\setminus B_{2\eps}^g$, we have 
	$$\int_\Om \varphi_\eps^2 u^2 \, d\mu_g \leq \|u\|_{L^\infty(\Om)}^2 |B_{2\eps}^g|_f\to 0 \quad \text{as} \ \eps \to 0 \, .$$
	Moreover, since $|\nabla \varphi_\eps|_g\leq \frac 2 \eps$ and $|\nabla u|_g^2 \in L^\infty(\Om)$, we obtain that
	\begin{align*}
		\int_\Om |\nabla (u\varphi_\eps)|_g^2 \, d\mu_g &= \int_\Om |\nabla\varphi_\eps|_g^2 \,  u^2 \, d\mu_g+\int_\Om \varphi_\eps^2 \, |\nabla u|_g^2 \, d\mu_g\\
		&\leq 4\|u\|_{L^\infty(\Om)}^2\, \frac{|B_{2\eps}^g|_f}{\eps^2}+\||\nabla u|_g\|_{L^\infty(\Om)}^2|B_{2\eps}^g|_f \to 0 \quad \text{as} \ \eps \to 0 \,.
	\end{align*}
\end{proof}
\begin{remark}
	Note that for $n=2$, the parameter ranges defined by \eqref{alphagamma} and \eqref{alphagamma1} do not overlap, resulting in an empty intersection. Therefore, the proof cannot be concluded in this case.
\end{remark}
\begin{remark}\label{gammamin1}
	Note that the assumption $\gamma<1$ is not restrictive. Indeed, for $\gamma>1$ one can still obtain a Serrin-type rigidity result for the solution $u$ to \eqref{ourprob} under the assumptions of \Cref{main}, for a suitable range of $\alpha$ and $\gamma$. More precisely, \Cref{Ofinitedist} implies that, if $\gamma<1$, then the origin lies at finite distance. Analogously, one can prove that, in this case,
	$$d_g(x,+\infty)=+\infty \quad \forall x \in \R^n\setminus\{O\} \, ,$$
	where
	$$
	d_g(x,+\infty)\coloneq \inf\left\{L(\sigma): \sigma:[0,1)\to \R^n\setminus\{O\},\ \sigma(0)=x,\ \lim_{t\to1^-}|\sigma(t)|=+\infty\right\}\, .
	$$  
	Arguing in the same way, one can observe that, if $\gamma>1$, the roles of the origin and infinity are interchanged. In particular, 
	$$
	d_g(x,O)=+\infty \quad \forall x \in \mathbb{R}^n\setminus\{O\}
	$$
	and
	\begin{equation}\label{distinfty}
		d_g(x,+\infty)=\frac{|x|^{1-\gamma}}{\gamma-1}<+\infty \quad \forall x \in \mathbb{R}^n\setminus\{O\}\, 
	\end{equation}
	Hence, in this case, infinity lies at finite distance. Accordingly, the weight becomes 
	\begin{equation}\label{weightgammageq1}
		e^{-f}\coloneq d_g(x,+\infty)^\alpha \, .
	\end{equation}
	For $R>0$, define
	$$
	B_R^g(\infty)\coloneq \left\{x \in \mathbb{R}^n\setminus\{O\}: d_g(x,+\infty)<R\right\}\, .
	$$
	From \eqref{distinfty}, it is straightforward to verify that
	\begin{equation}\label{ballinfty}
		B_R^g(\infty)=\mathbb{R}^n\setminus \overline{B_{R^*}^{\delta}}, 
		\qquad \text{with } 
		R^*\coloneq \left(\frac{1}{R(\gamma-1)}\right)^{\frac{1}{\gamma-1}}.
	\end{equation}
	\par Let $\Omega$ and $u$ be a solution to \eqref{ourprob}, with $e^{-f}$ defined as in \eqref{weightgammageq1}, satisfying the assumptions of \Cref{main}.
	The argument then proceeds as in the case $\gamma<1$, but leads to a different admissible range of parameters. More precisely, the range in \eqref{parcond} is determined by imposing the following conditions:
	\begin{enumerate}
		\item [(i)] $e^{-f}\in L^1(\Omega,dV_g)$;
		\item [(ii)] \eqref{condgammadiv}, in order to apply \Cref{lemmadivthm};
		\item [(iii)] \eqref{kpositive}, ensuring positivity of the weight $k$ defined in \eqref{k_nc} and used in the Young inequality \eqref{youngw_nc};
		\item [(iv)] \eqref{alphagamma}, which implies  \eqref{matrixgeq0}.
	\end{enumerate}
	From \eqref{ballinfty}, one easily checks that
	$$
	e^{-f}\in L^1(\Omega,dV_g) \iff \alpha>-n,
	$$
	and \Cref{lemmadivthm} holds provided that $\alpha>2-n$. Hence, conditions (i)--(iii) are satisfied when $\alpha>0$. Combining this with (iv), we obtain
	$$
	\alpha>0 \quad \text{and} \quad 1<\gamma<\gamma_2,
	\quad \text{with} \quad 
	\gamma_2 \coloneq 1+\sqrt{\frac{n-2}{\alpha+n-2}}.
	$$
	Hence, arguing as in the case $\gamma<1$, we conclude that $u$ is depends only on $d_g(x,+\infty)$, it is given by
	$$
	u(x)=\frac{R^2-d_g(x,+\infty)^2}{2(n+\alpha)}\, ,
	$$
	and
	$$\Omega = B_R^g(\infty) \, ,$$
	for some $R>0$.
\end{remark}
	\hfill\\ 
	\par In the compact case, we provided an alternative proof of \Cref{main} based on the $P$-function method, which relies on the strong maximum principle. More precisely, we considered the function  
	\begin{equation}\label{Pfunction_rem}
		P(u)=|\nabla u|_g^2 + \frac{2\beta}{n}\, u \, ,
	\end{equation}
	defined on $M=\overline{\Omega}^g$.
	Under the assumptions of \Cref{main}, if $u$ is a solution to \eqref{ourprob}, then $\lapw P(u)\geq 0$. By the strong maximum principle applied on $M$, $P(u)$ cannot attain a maximum in $\Omega$ unless it is constant throughout $\Omega$. Consequently, either $P(u)$ is constant in $\Omega$, or
	\begin{equation}\label{Pmaxprinc}
		P(u) < \sup_{\overline{\Omega}^g} P(u) \quad \text{in } \Omega \, .
	\end{equation}
	Since $\overline{\Omega}^g$ is compact and $P(u)\in C^0(\overline{\Omega}^g)$, from \eqref{Pmaxprinc} it follows that
	$$
	P(u) < \max_{\partial \Omega} P(u) \quad \text{in } \Omega \, .
	$$
	Using the boundary condition $u=0$ on $\partial \Omega$, we deduce that
	\begin{equation}\label{Pleqc2_complete}
		P(u) < c^2 \quad \text{in } \Omega \, .
	\end{equation}
	Finally, combining \eqref{Pleqc2_complete} with the Pohozaev identity, we showed that \eqref{Pleqc2_complete} cannot hold, and therefore $P(u)$ must be constant in $M$.
	\par In contrast, in the non-compact setting the manifold $M=\overline{\Omega}^g$ is no longer compact if $O \in \overline{\Omega}$ (see \Cref{preliminarynoncompl}). The strong maximum principle still implies that $P(u)$ cannot attain a maximum in $\Omega$ unless it is constant in $\overline{\Omega}^g$. Hence, if this is not the case, \eqref{Pmaxprinc} remains valid. However, due to the lack of compactness, $P(u)$ may fail to achieve a maximum on $\overline{\Omega}^g$. In particular, from \eqref{Pmaxprinc} we only obtain
	\begin{equation*}
		\sup_{\overline{\Omega}^g} P(u)
		= \max\left\{ \max_{\partial \Omega} P(u), \limsup_{x \to O} P(u(x)) \right\}.
	\end{equation*}
	Even if we assume additional regularity on $u$ so that $P(u)$ extends continuously to the origin, \eqref{Pmaxprinc} reduces to
	$$
	P(u) < \max\{ c^2, P(u(O)) \} \quad \text{in } \Omega \, .
	$$
	However, we have no a priori control on the value $P(u(O))$, and therefore we cannot conclude the argument by means of the Pohozaev identity as in the compact case.
	
	Thus, the proof based on the $P$-function method breaks down in the non-compact setting.
	This observation highlights the importance of our approach, which avoids the use of maximum principles and remains applicable in the non-compact setting.
\subsection{Failure of the rigidity result: an example}\label{failuresect}
In this section, we construct explicit examples of a Riemannian metric of the form \eqref{g} for which one cannot expect the ovederdetermined problem \eqref{ourprob} to have the ball as the solution. For simplicity of notation, we denote the weight function 
\begin{equation}\label{wrad}
	w \coloneqq e^{-f}=d_g(x,O)^\alpha \, .
\end{equation}
Let $\gamma<1$ and $u=u(d_g(x,O))$ be a radial solution to the equation
\begin{equation}\label{eq}
	\lapw u=\lap_g u+g(\nabla \log w,\nabla u)=-1 \quad \text{in} \ \Om \, .
\end{equation}
If $r(x) \coloneqq d_g(x,O) = \frac{|x|^{1-\gamma}}{1-\gamma}$, then straightforward computations show that
\begin{equation}\label{gradrad}
	Du = u' \frac{x}{|x|^{1+\gamma}}
\end{equation}
and, from \eqref{lap}
\begin{equation}\label{laprad}
	\lap_g u = \left(u''+\frac{(n-1)}{r}u'\right) \, ,
\end{equation}
where the primes $'$ and $''$ denote derivatives with respect to the radial variable $r$.\\
Since $w$ is radial, it follows from \eqref{gradrad} that
\begin{equation}\label{gradcdot}
	g(\nabla \log w,\nabla u)= \frac{w'u'}{w} \, .
\end{equation}
Substituting \eqref{gradcdot} and \eqref{laprad} into \eqref{eq}, we find that any $u$ radial solution to $\lapw u=-1$ must satisfy
\begin{equation}\label{eqradial}
	\left\{u''+\left(\frac{(n-1)}{r} +\frac{ w'}{w}\right)u'\right\} =-1 \quad \text{in} \ \Om \, .
\end{equation}
We now look for a radial solution to \eqref{ourprob} defined on an annulus
$$\Om \coloneq \{x \in \R^n\setminus \{O\} : a<r(x)<b\} \, ,$$
with $0<a<b$. Note that, in this case, the weight $w$ defined in \eqref{wrad} belongs to $L^1(\Om,dV_g)$ for every $\alpha \in \R$. Recalling \eqref{ballgeucl}, the ball $B_R^g$ centered at the origin corresponds to a Euclidean ball centered at the origin with a suitable radius. Therefore, the Euclidean outward unit normal vector to $\p\Om$ at $r=b$ is
$$\nu_e = \frac{x}{|x|} \, .$$
The corresponding outward unit normal vector with respect to the metric $g$ is
$$\nu_{b,i} =\frac{g^{ij}\nu_j}{(\nu_k g^{kl}\nu_l)^{1/2}}=\frac{x_i}{|x|^{1-\gamma}} \, .$$
Note that at $r=a$, the outward normal is $\nu_a=-\nu_b$.\\
Hence, using \eqref{gradrad} and \eqref{eqradial}, we see that $u$ must satisfy  
\begin{equation}\label{radprob}
	\begin{cases}
		u''+\left(\frac{(n-1)}{r} +\frac{w'}{w}\right)u'=-1 & \text{if} \ r \in (a,b) \, , \\
		u(a)=u(b)=0 \, ,\\
		u'(b)=-u'(a)=-c \, .
	\end{cases}
\end{equation}
Note that \eqref{eqradial} yields
$$(r^{n-1}w u')' =r^{n-1}w \left(u''+\left(\frac{n-1}{r}+\frac{w'}{w}\right)u'\right)=-r^{n-1}w \, .$$
Integrating, we obtain 
\begin{equation}\label{derurad}
	u_r(r)=C-\frac{1}{r^{n-1}w(r)}\int_a^r s^{n-1}w(s)\, ds \, ,
\end{equation}
where $C=a^{n-1}w(a)u'(a)$. \\
The boundary conditions $u(a)=u(b)=0$ imply 
$$\int_a^b u_r(r) \, dr =0 \, .$$
Substituting \eqref{derurad} into this identity gives
\begin{equation}\label{C}
	C \int_a^b \frac{1}{r^{n-1}w(r)} \, dr -\int_a^b \frac{1}{r^{n-1}w(r)} \left(\int_a^r s^{n-1}w(s) \, ds \right) \, dr =0 \, .
\end{equation}
Moreover, using \eqref{derurad}, the overdetermined condition $u'(b)=-u'(a)$ becomes
\begin{equation}\label{condneurad}
	C \left(\frac{1}{b^{n-1}w(b)}+\frac{1}{a^{n-1}w(a)}\right)=\frac{1}{b^{n-1}w(b)}\int_a^b s^{n-1}w(s) \, ds \, .
\end{equation}
Combining \eqref{C} and \eqref{condneurad}, we obtain the compatibility condition
\begin{equation}\label{compcond1}
	\begin{aligned}
		&\left(\frac{1}{b^{n-1}w(b)}+\frac{1}{a^{n-1}w(a)}\right)\int_a^b \frac{1}{r^{n-1}w(r)} \left(\int_a^r s^{n-1}w(s) \, ds \right) \, dr\\
		&= \frac{1}{b^{n-1}w(b)}\int_a^b \frac{1}{r^{n-1}w(r)} \, dr \int_a^b s^{n-1}w(s) \, ds \, .
	\end{aligned}
\end{equation}
Recalling that $w(r)=r^\alpha$ and defining 
\begin{equation}\label{beta}
	\beta \coloneq n+\alpha-1 \, ,
\end{equation}
we can rewrite \eqref{compcond1} as
\begin{equation}\label{compcond}
	\left(\frac{1}{b^\beta}+\frac{1}{a^\beta}\right)\int_a^b \frac{1}{r^\beta} \left(\int_a^r s^\beta \, ds \right) \, dr = \frac{1}{b^\beta}\int_a^b \frac{1}{r^\beta} \, dr \int_a^b s^\beta \, ds \, .
\end{equation}
Assume now that $\beta \neq -1$ and $\beta\neq 1$. These cases will be discussed later. Evaluating the integrals in \eqref{compcond} yields
\begin{align*}
	&\frac{1}{\beta+1}\left\{\left(\frac{1}{a^\beta}+\frac{1}{b^\beta}\right)\left(\frac{b^2-a^2}{2}\right)-\left(a-\frac{a^{\beta+1}}{b^\beta}\right)\int_a^b r^{-\beta}\right\}\\
	&=\frac{1}{\beta+1}\left(b-\frac{a^{\beta+1}}{b^\beta}\right)\int_a^b r^{-\beta}\, dr \, .
\end{align*}
Rearranging the terms, we obtain that
\begin{equation*}
	\int_a^b r^{-\beta} \, dr =\frac{b^2-a^2}{2(a+b)}\left(\frac{1}{a^\beta}+\frac{1}{b^\beta}\right) \, .
\end{equation*}
Since $\beta \neq 1$, we get
\begin{equation}\label{compcondbeta}
	\frac{b^{1-\beta}-a^{1-\beta}}{1-\beta}=\frac{b-a}{2}\left(\frac{1}{a^\beta}+\frac{1}{b^\beta}\right) \, .
\end{equation}
Introducing the variable $\rho \coloneq b/a >1$, equation \eqref{compcondbeta} becomes
\begin{equation}\label{compcondx}
	\frac{\rho^{1-\beta}-1}{1-\beta}=\frac{\rho-1}{2}(\rho^{-\beta}+1) \, .
\end{equation}
Thus, \eqref{compcondx} is the final form of the compatibility condition \eqref{compcond} for $\beta\neq -1,1$, that is, for weights $w(r)=r^\alpha$ with $\alpha \neq -n,2-n$. Fixing any $\alpha\in \R\setminus\{-n,2-n\}$, the existence of a solution $u$ to \eqref{radprob} requires that \eqref{compcondx} is satisfied for some $\rho=b/a>1$, where $\beta =n+\alpha-1$.\\
To determine for which values of $\beta$ this condition holds, we define the function 
$$F_\rho(\beta)\coloneq \frac{\rho^{1-\beta}-1}{1-\beta}-\frac{\rho-1}{2}(\rho^{-\beta}+1) \, .$$
Then, \eqref{compcondx} is satisfied if and only if $F_\rho(\beta)=0$. Now we show that, for any fixed $\rho>1$, $F_\rho(\beta)=0$ if and only if $\beta=-1$ or $\beta =0$. \\
Recall that $\beta \neq -1,1$. Consider the function
$$
f_\beta(t)=t^{-\beta} \, , \quad t\in[1,\rho] \, .
$$
A direct computation shows that
\begin{equation}\label{F1}
	\int_1^\rho f_\beta(t)\,dt
	=\int_1^\rho t^{-\beta}\,dt
	=\frac{\rho^{1-\beta}-1}{1-\beta} \, .
\end{equation}
Moreover, note that
\begin{equation}\label{F2}
	\frac{\rho-1}{2}\left(1+\rho^{-\beta}\right)=\frac{\rho-1}{2}\bigl(f_\beta(1)+f_\beta(\rho)\bigr) \, .
\end{equation}
Hence, using \eqref{F1} and \eqref{F2}, from the definition of $F_\rho(\beta)$ we have
$$
F_\rho(\beta)
=\int_1^\rho f_\beta(t)\,dt
-\frac{\rho-1}{2}\bigl(f_\beta(1)+f_\beta(\rho)\bigr) \, .
$$
Let $L_\beta(t)$ denote the secant line joining the points
$(1,f_\beta(1))$ and $(\rho,f_\beta(\rho))$. Since
$$
\int_1^\rho L_\beta(t)\,dt
=\frac{\rho-1}{2}\bigl(f_b(1)+f_b(\rho)\bigr) \, ,
$$
we may write
$$
F_\rho(\beta)=\int_1^\rho \bigl(f_\beta(t)-L_\beta(t)\bigr)\,dt \, .
$$
We now analyze the sign of $F_\rho(\beta)$. A straightforward computation yields
$$
f_\beta''(t)=\beta(\beta+1)\,t^{-\beta-2} \, .
$$
Since $t^{-\beta-2}>0$ for all $t>0$, the sign of $f_\beta''$ is determined by $\beta(\beta+1)$:
\begin{itemize}
	\item If $\beta>0$ or $\beta<-1$, then $\beta(\beta+1)>0$ and $f_\beta$ is strictly convex.
	\item If $-1<\beta<0$, then $\beta(\beta+1)<0$ and $f_\beta$ is strictly concave.
	\item If $\beta=0$ or $\beta=-1$, then $f_\beta''\equiv 0$ and $f_\beta$ is affine.
\end{itemize}
If $f_\beta$ is strictly convex, then $f_\beta(t)<L_\beta(t)$ for all $t\in(1,\rho)$,
which implies
$$
F_\rho(\beta)<0 \qquad \text{for } \beta>0 \text{ or } \beta<-1 \, .
$$
If $f_\beta$ is strictly concave, then $f_\beta(t)>L_\beta(t)$ for all $t\in(1,\rho)$,
which implies
$$
F_\rho(\beta)>0 \qquad \text{for } -1<\beta<0.
$$
Thus, for every fixed $\rho>1$, the function $F_\rho(\beta)$ has zeros only at
$$
\beta=0 \quad\text{and}\quad \beta=-1.
$$
We have excluded the case $\beta=-1$.\\
If $\beta=-1,1$, one can check directly from \eqref{compcond} that the compatibility condition reduces to
$$2t \ln t =t^2-1$$
for some $t=b/a>1$. It is not difficult to see that this equation has no solution for any $t>1$.\\
We therefore focus on the remaining admissible case $\beta=0$. From \eqref{beta} this corresponds to $\alpha=1-n$, and hence,
 $$w=d_g(x,O)^{1-n} \, .$$
Note that this choice of parameters does not belong to the admissible range \eqref{parcond}. In this case, we have
$$\frac{w'}{w}=-\frac{n-1}{r} \, .$$
Thus, it can be directly verified from equation \eqref{eqradial} that the function
$$u(x) = \frac{1}{2}(r(x)-a)(b-r(x))$$
defines a (smooth) solution to the equation $\lapw u=-1$ in the annular domain $\Om \coloneqq \{x \in \R^n\setminus \{O\}: a<r(x)<b\}$, where $a,b>0$.
\par This shows that \Cref{main_nc} cannot be extended to arbitrary values of the parameters $\alpha$ and $\gamma$. The result highlights the delicate nature of rigidity: while it holds within a specific parameter regime, outside this range it may fail, giving rise to more domain geometries. Moreover, our analysis shows that this is the only choice of parameters for which a radial solution of \eqref{ourprob} exists on an annulus.\\
Finally, if $\gamma = 0$ (i.e., in the Euclidean setting) note that this provides a partial answer to the open problem discussed in \Cref{ex_power}: we have exhibited a weight of the form $|x|^\alpha$ for which the existence of a solution to the weighted overdetermined problem \eqref{ourprob} does not imply that the domain is a ball. It is worth noting that, in this case, condition \eqref{Ricf_hp} is satisfied for $N\leq N_\beta=n/\beta$. Indeed, as noted in \Cref{ex_power}, we have $$\beta=\frac{n}{n+\alpha}$$
and so
$$N_\beta = \frac n \beta = n+\alpha \, .$$
Thus, from \eqref{Ricex1} we find that \eqref{Ricf_hp} holds. This shows that a condition of this type is non-restrictive: by itself, it does not imply that the existence of a radial solution to the overdetermined problem forces the domain to be a ball.
\section*{Acknowledgments}
The authors thank Luciano Mari for many useful discussions about this project. {The authors are members of the Gruppo Nazionale per l'Analisi Matematica, la Probabilit\`a e le loro Applicazioni (GNAMPA) of the Istituto Nazionale di Alta Matematica (INdAM, Italy). The second author has been partially supported by the INdAM - GNAMPA Project 2026, CUP \#E53C25002010001\#.}

\appendix
\crefalias{section}{appendix}
\section{}\label{appendix}
In this appendix we provide a proof of \Cref{farinaroncoroni}.
\begin{proof}[Proof of \Cref{farinaroncoroni}]
From \cite[Lemma 6]{FarinaRoncoroni} and its proof it follows that the domain $\Omega$ is a metric ball $B_R^g(p)$ centered at a maximum point $p$ of $u$ and the function $u$ is radial.

Recall that the injectivity radius at $p$, denoted $\inj(p)$, is the largest $\varepsilon>0$ such that the exponential map
$$
\exp_p : B(0,\varepsilon) \subset T_p M \longrightarrow B_\eps^g(p) \subset M
$$
is defined and is a diffeomorphism. Since $\overline\Om^g$ is compact by assumption, it follows that $\inj(p)>0$. Let
\begin{equation}\label{eps}
	\varepsilon = \min\{R, \inj(p)\} \, .
\end{equation}
We use the exponential map to introduce geodesic polar coordinates around $p$. Let $S^{n-1} \subset T_p M$ be the unit sphere, and fix an orthonormal basis $\{e_i\}_{i=1}^n$ of $T_p M$. For a vector $v = v^i e_i \in T_p M$, define its radial coordinate by
$$
\rho(v) = \sqrt{\sum_{i=1}^n (v^i)^2} \, ,
$$
and let $\{\psi^\alpha\}_{\alpha=1}^{n-1}$ be a coordinate system on $S^{n-1}$. Then $(\rho, \psi^1, \dots, \psi^{n-1})$ form coordinates on $(0,\varepsilon) \times S^{n-1} \subset T_p M$.
Through the exponential map, these induce coordinates $(r, \theta^1, \dots, \theta^{n-1})$ on $B_\eps^g(p)$ defined by
\begin{equation}\label{coordinates}
	r \coloneq \rho \circ \exp_p^{-1} \quad \text{and} \quad \theta^\alpha \coloneq \psi^\alpha \circ \exp_p^{-1}, \quad \alpha=1,\dots,n-1.
\end{equation}
In these coordinates, the metric takes the form
\begin{equation*}
	g = dr \otimes dr + g_{\alpha\beta}(r,\theta) \, d\theta^\alpha \otimes d\theta^\beta \, ,
\end{equation*}
with 
\begin{equation}\label{condg}
	\frac{g_{\alpha\beta}(r,\theta)}{r^2}\to \overline g_{\alpha\beta}(\theta) \quad \text{as} \ r\to 0 \, ,
\end{equation}
where $\overline g_{\alpha\beta}(\psi)d\psi^\alpha\ot d\psi^\beta$ is the standard metric on $S^{n-1}$ induced by $\R^n$.\smallskip
\par Our strategy aims first to show that
$$
g_{\alpha\beta}(r,\theta) = \mathrm{sn}_k^2(r) \, \overline{g}_{\alpha\beta}(\theta) \quad \text{in } B_\eps^g(p) \, .
$$
Next, we prove that $\exp_p$ remains a diffeomorphism on all of $\Omega$. This implies that $\inj(p) \geq R$, so that $\varepsilon = R$. Consequently, we conclude that
$$
g_{\alpha\beta}(r,\theta) = \mathrm{sn}_k^2(r) \, \overline{g}_{\alpha\beta}(\theta) \quad \text{in } B_R^g(p)=\Om \, ,
$$
from which it follows that $\Omega$ is isometric to a metric ball in the space form $S^n_k$.\smallskip\\
\textbf{Step 1:} $p$ is the only critical point of $u$.

Let $q \in \overline\Omega^g$ with $q\neq p$. By \cite[Proposition 12]{FarinaRoncoroni}, there exists a minimizing and unit speed geodesic 
$\gamma_q : [0,1] \to \Omega$ such that $\gamma_q(0) = p$ and $\gamma_q(1) = q$. Define
\begin{equation}\label{ft}
	f(t) := u(\gamma_q(t)) \, .
\end{equation}
Note that
\begin{equation}\label{f'g}
	f'(t) = g(\nabla u(\gamma_q(t)), \dot{\gamma}_q(t)) \, .
\end{equation}
Moreover, $f$ satisfies
\begin{align*}
	f''(t)&=\frac{d^2}{dt^2}(u\circ \gamma_q)(t)\\
	&=\frac{d}{dt}g(\nabla u(\gamma_q(t)),\dot\gamma (t))\\
	&=g\big(\nabla_{\dot\gamma_q}\nabla u(\gamma_q(t)),\dot\gamma_q(t)\big)+g\big(\nabla u(\gamma_q(t)),\nabla_{\dot\gamma_q}\dot\gamma_q(t)\big)\\
	&=g\big((\nabla_{\dot\gamma_q(t)}\nabla u)(\gamma_q(t)),\dot\gamma_q(t)\big)\\
	&=\nabla^2 u\rvert_{\gamma_q(t)}(\dot\gamma_q(t),\dot\gamma_q(t)) \, .
\end{align*}
Then, from the equation in \eqref{eqhessu} we obtain that
\begin{equation}\label{f_geod}
	f''(t) = -\left( \frac{1}{n} + k \, f(t) \right) \, .
\end{equation}
Since $p$ is a maximum point for $u$, evaluating \eqref{f'g} at $t=0$ gives
\begin{equation}\label{f'0}
	f'(0) = 0 \, .
\end{equation}
Hence, integrating \eqref{f_geod} and imposing \eqref{f'0} yields
\begin{equation}\label{f'}
	f'(t) = -\frac{t}{n} - k \int_0^t f(s) \, ds \, .
\end{equation}
Recall that $u$ is positive in $\Omega$, and thus
\begin{equation*}
	f(s) > 0 \, , \quad s \in [0,1).
\end{equation*}
Therefore, if $k \ge 0$, it follows immediately from \eqref{f'} that $f'(t) = 0$ if and only if $t = 0$. \\
Now assume $k < 0$.
Solving \eqref{f_geod} with initial conditions $f(0) = u(p)$ and $f'(0) = 0$, we obtain
\begin{equation*}
	f(t) = \left( u(p) + \frac{1}{nk} \right) \cosh(\sqrt{-k} \, t) - \frac{1}{nk} \, ,
\end{equation*}
and therefore
\begin{equation*}
	f'(t) = \left( u(p) + \frac{1}{nk} \right) \sqrt{-k} \, \sinh(\sqrt{-k} \, t) \, ,
\end{equation*}
which vanishes if and only if $t = 0$. Indeed, we must have $\left( u(p) + \frac{1}{nk} \right) \neq 0$ (its value will be determined later by imposing the boundary condition $u=0$ on $\p\Om$). Otherwise, taking $q \in \p\Om$ and recalling \eqref{ft}, we would obtain in particular 
$$f(1)=u(q)=-\frac{1}{nk}\neq 0 \, ,$$
which contradicts the boundary condition $u=0$ on $\p\Om$.\\
Hence, since $f'(t)=0$ if and only if $t=0$, from \eqref{f'g} we conclude that $p$ is the only critical point of $u$.\smallskip\\
\textbf{Step 2:} $g_{\alpha\beta}(r,\theta)=\mathrm{sn}_k^2(r) \overline g_{\alpha\beta}(\theta)$ in $B_\eps^g(p)$.\\
Recall that $r$ coincides with the distance from $p$ in $B_\eps^g(p)\setminus\{p\}$. Hence, since $u$ is radial, we have that
\begin{equation}\label{nablaur}
	\nabla u=u'(r)\nabla r \quad \text{in } B_\eps^g(p)\setminus\{p\} \, .
\end{equation}
Recalling that $|\nabla r|_g=1$, we get
$$\frac{\nabla u}{|\nabla u|_g}=\sign(u'(r))\nabla r \, .$$ 
From Step 1, we know that $p$ is the unique maximum point of $u$, so $u$ must decrease along any geodesic moving away from $p$. It follows that $u'(r)<0$. Therefore, we conclude that
\begin{equation}\label{nablar}
	\frac{\nabla u}{|\nabla u|_g}=-\nabla r \quad \text{in} \ B_\eps^g(p)\setminus\{p\} \, .
\end{equation}
Let $X,Y$ be smooth vector fields on $B_\eps^g(p)$. Note that, using \eqref{nablaur}, we have
\begin{align*}
	\nabla^2u(X,Y)&=\nabla(u'dr)(X,Y)\\
	&=\nabla_X(u'dr)(Y)\\
	&=X(u')dr(Y)+u'\nabla_X(dr)(Y) \, ,
\end{align*}
that is,
$$\nabla^2u(X,Y)=X(u')dr(Y)+\nabla^2 r(X,Y) \, .$$
Thus, taking first $X=Y=\p_r$, then $X=\p_{\theta^\al}$, $Y=\p_{\theta^\bb}$ and finally $X=\p_r$, $Y=\p_{\theta^\al}$, we obtain respectively
$$\nabla^2 u(\p_r,\p_r)=u''\, ,$$
\begin{align*}
	\nabla^2 u(\p_{\theta^\al},\p_{\theta^\bb})=u'\nabla^2r(\p_{\theta^\al},\p_{\theta^\bb})&=\frac 1 2 u' g(\nabla_{\p_{\theta^\al}}\p_r,\p_{\theta^\bb})\\
	&=\frac 1 2 u' g(\nabla_{\p_r}\p_{\theta^\al},\p_{\theta^\bb})\\
	&=\frac 1 2 u' \p_r g_{\al\bb} 
\end{align*}
and 
$$\nabla^2 u(\p_r,\p_{\theta^\al})=0 \, .$$
Therefore, we conclude that 
\begin{equation}\label{hessur}
	\nabla^2 u=u'' dr \ot dr + \frac 1 2 u' \p_r g_{\al\bb} d\theta^\al\ot d\theta^\bb \quad \text{in } B_\eps^g(p)\, .
\end{equation}
Comparing \eqref{hessur} with \eqref{eqhessu}, we deduce that
\begin{equation}\label{equ''}
	u''=-\left(\frac 1 n +ku\right)\, , \quad r \in (0,\eps)
\end{equation}
and 
\begin{equation}\label{eqg}
	\p_r g_{\al\bb}=-\frac{2}{u'}\left(\frac 1 n +ku\right)g_{\al\bb}\, , \quad r\in(0,\eps)\, .
\end{equation}
Solving \eqref{equ''} with the initial conditions $u'(0)=0$ and $u(0)=u(p)$, we find that in $B_\eps^g(p)$ 
\begin{equation}\label{uBeps}
	u(r)=\begin{cases}
		\left(u(p)+\frac{1}{nk}\right) \cos(\sqrt{k}r)-\frac{1}{nk} & \text{if } k>0\, ,\\
		u(p)-\frac{r^2}{2n} & \text{if } k=0 \, ,\\
		\left(u(p)+\frac{1}{nk}\right)\cosh(\sqrt{-k}r)-\frac{1}{nk} & \text{if } k<0 \, .
	\end{cases}
\end{equation}
It follows that
\begin{equation}\label{u'}
	u'(r)=\begin{cases}
		-\left(u(p)+\frac{1}{nk}\right) \sqrt{k}\sin(\sqrt{k}r) & \text{if } k>0\, ,\\
		-\frac{r}{n} & \text{if } k=0 \, ,\\
		\left(u(p)+\frac{1}{nk}\right)\sqrt{-k}\sinh(\sqrt{-k}r) & \text{if } k<0 \, .
	\end{cases}
\end{equation}
Finally, substituting \eqref{uBeps} and \eqref{u'} into \eqref{eqg}, solving the resulting equation and imposing the condition \eqref{condg}, we obtain
\begin{equation}\label{formgalb}
	g_{\al,\bb}(r,\theta)=\mathrm{sn}_k^2(r) \, \overline{g}_{\alpha\beta}(\theta) \quad \text{in } B_\eps^g(p) \, .
\end{equation}
\textbf{Step 3:} $\inj(p) \geq R$.\\
Let $v \in T_pM$ with $|v|_g=\inj(p)$. Then, either
\begin{equation}\label{caseexp1}
	(\exp_p)_{*,v} \text{ is not an isomorphism}
\end{equation}
or
there exists $v_1 \in T_pM$, $v_1\neq v$ with $|v_1|_g=|v|_g$ such that
\begin{equation}\label{caseexp2}
	\exp_p(v)=\exp_p(v_1)\eqqcolon x \text{ and } \left. \frac{d}{dt} \right\rvert_{t = 1} \exp_p(tv)=-\left. \frac{d}{dt} \right\rvert_{t = 1} \exp_p(t v_1) \, ,
\end{equation}
as shown, for example, in \cite[Lemma 5.7.12]{Petersen}.
By contradiction, assume that $|v|_g<R$. Then, recalling \eqref{eps}, we have
$$\eps=|v|_g \, .$$ 
\par We first show that \eqref{caseexp1} cannot occour. Note that
$$(\exp_p)_{*,v}:T_v(B(0,|v|_g))\to T_{\exp_p(v)}(B_{|v|_g}^g(p))\, ,$$
and
$$ \left(\left.\frac{\p}{\p\rho}\right\rvert_{v},\left.\frac{\p}{\p\psi^1}\right\rvert_{v},\dots,\left.\frac{\p}{\p\psi^{n-1}}\right\rvert_{v}\right)$$
is a basis of $T_v(B(0,|v|_g))$.\\
From \eqref{formgalb}, the pull-back metric $(\exp_p)^* g$ induced on $B(0,|v|_g)$ has the form 
\begin{equation}\label{formgtpm}
	(\exp_p)^* g=d\rho\ot d\rho +\mathrm{sn}_k^2(\rho)\, \overline g_{\al\bb}(\psi)d\psi^\al \ot d\psi^\bb \, .
\end{equation}
In particular, 
$$\left.\frac{\p}{\p\rho}\right\rvert_{v} \notin \ker((\exp_p)_{*,v}) \, ,$$
since \eqref{formgtpm} implies
$$1=((\exp_p)^* g)\left(\left.\frac{\p}{\p\rho}\right\rvert_{v},\left.\frac{\p}{\p\rho}\right\rvert_{v}\right)=g\left((\exp_p)_{*,v}\left(\left.\frac{\p}{\p\rho}\right\rvert_{v}\right),(\exp_p)_{*,v}\left(\left.\frac{\p}{\p\rho}\right\rvert_{v}\right)\right)$$
Hence, we can choose coordinates $\{\psi^\alpha\}_\al$ such that 
\begin{equation}\label{kerexp*}
	\left.\frac{\p}{\p\psi^\al}\right\rvert_{v} \in \ker((\exp_p)_{*,v})
\end{equation}
for some $\alpha \in \{1,\dots,n-1\}$.\\
Let $t \in (0,1)$. Since $tv \in B(0,|v|_g)$, from \eqref{formgtpm} we obtain
\begin{equation}\label{gaasin}
	g_{\alpha\alpha}\big(\rho(tv),\theta(tv)\big)=\mathrm{sn}_k^2(\rho(tv))\,\overline g_{\alpha\al}(\psi(tv))=\mathrm{sn}_k^2(t|v|_g)\,\overline g_{\alpha\al}(\psi(tv)) \, .
\end{equation}
On the other hand, 
\begin{equation}\label{gaa0}
	\begin{aligned}
		g_{\alpha\alpha}(\rho(tv),\psi(tv))&=((\exp_p)^* g)\left(\left.\frac{\p}{\p\psi^\al}\right\rvert_{tv},\left.\frac{\p}{\psi^\al}\right\rvert_{tv}\right)\\
		&=g\left((\exp_p)_{*,tv}\left(\left.\frac{\p}{\p\psi^\al}\right\rvert_{tv}\right),(\exp_p)_{*,tv}\left(\left.\frac{\p}{\p\psi^\al}\right\rvert_{tv}\right)\right) \, .
	\end{aligned}
\end{equation}
Comparing \eqref{gaa0} and \eqref{gaasin}, we get
\begin{equation*}
	g\left((\exp_p)_{*,tv}\left(\left.\frac{\p}{\p\psi^\al}\right\rvert_{tv}\right),(\exp_p)_{*,tv}\left(\left.\frac{\p}{\p\psi^\al}\right\rvert_{tv}\right)\right)=\mathrm{sn}_k^2(t|v|_g)\overline g_{\alpha\al}(\psi(tv)) \, .
\end{equation*}
Thus, letting $t\to 1$ and using \eqref{kerexp*}, we obtain that
$$0=\mathrm{sn}_k^2(\rho(v))\overline g_{\al\al}(\psi(v)) \, .$$
Since $\overline g_{\al\al}(\psi(v))>0$, it follows that
$$\mathrm{sn}_k^2(\rho(v))=\mathrm{sn}_k^2(|v|_g)=0 \, .$$
This implies
$$|v|_g=0 \, ,$$
which is impossible, since $|v|_g=\inj(p)>0$.\\
\par We now show that \eqref{caseexp2} cannot occur either. Define
$$\gamma_j(t)\coloneq \exp_p(tv_j)\, , \quad j \in \{1,2\}\, ,$$
where $v_0\coloneq v$. Since each $\gamma_j$ is a geodesic starting from $p$, we have 
$$\dot\gamma_j(t)=|\dot\gamma_j|_g\nabla r(\gamma_j(t))\, , \quad t \in (0,1) \, .$$
Note that $r(\gamma_j(t))<|v|_g=r(x)$ and $|\dot\gamma_1|_g=|\dot\gamma_2|_g=|v|_g$. Then, using \eqref{nablar} we obtain that
$$\dot\gamma_j(t)=-|v|_g\frac{\nabla u}{|\nabla u|_g}(\gamma_j(t))\, , \quad t \in (0,1) \, .$$
By assumption, $r(x)=|v|_g<R$, so $\nabla u$ is continuous on $\overline B_{|v|_g}(p)$. Therefore, letting $t\to 1$ we conclude that
$$\dot\gamma_j(1)=-\frac{\nabla u}{|\nabla u|_g}(x)\, ,\quad j \in\{1,2\} \, .$$
This contradicts \eqref{caseexp2}, which asserts that $\dot\gamma_0(1)=-\dot\gamma_1(1)$.\\
\par Thus, it must hold
$$\inj(p)\geq R \, ,$$
which implies
\begin{equation}\label{epsR}
	\varepsilon = R \, .
\end{equation}
Hence, from \eqref{uBeps} we conclude that in $\Om$ $u$ takes the form
\begin{equation*}
	u(r)=\begin{cases}
		\left(u(p)+\frac{1}{nk}\right) \cos(\sqrt{k}r)-\frac{1}{nk} & \text{if } k>0\, ,\\
		u(p)-\frac{r^2}{2n} & \text{if } k=0 \, ,\\
		\left(u(p)+\frac{1}{nk}\right)\cosh(\sqrt{-k}r)-\frac{1}{nk} & \text{if } k<0 \, .
	\end{cases} 
\end{equation*}
Imposing the boundary condition $u(R)=0$ we determine $u(p)$, yielding
\begin{equation*}
	u(r)=\begin{cases}
		\frac{\cos(\sqrt{k}r)}{kn \cos(\sqrt{k}R)}-\frac{1}{nk} & \text{if } k>0\, ,\\
		\frac{R^2-r^2}{2n} & \text{if } k=0 \, ,\\
		\frac{\cosh(\sqrt{-k}r)}{kn\cosh(\sqrt{-k}R)}-\frac{1}{nk} & \text{if } k<0 \, .
	\end{cases} 
\end{equation*}
Finally, note that, if $k>0$, then $u$ is positive only if $$R<\frac{\pi}{2\sqrt{k}} \, .$$ 
Moreover, from \eqref{formgalb} and \eqref{epsR} we obtain
\begin{equation}\label{formgOm}
	g_{\al,\bb}(r,\theta)=\mathrm{sn}_k^2(r) \, \overline{g}_{\alpha\beta}(\theta) \quad \text{in } \Om \, .
\end{equation}
Thus, recalling \eqref{Rk}, we conclude that $\Om$ is isometric to a metric ball in the space form $S_k^n$. 
\end{proof}

\end{document}